\documentclass{article}
\usepackage{amsmath}
\usepackage{amsfonts}
\usepackage{amsthm}
\usepackage{amssymb}
\usepackage{color}
\usepackage{graphicx}
\usepackage{float}
\usepackage{subfigure}
\usepackage[font=small]{caption}
\usepackage{tikz}

\newtheorem{theorem}{Theorem}[section]
\newtheorem{lemma}[theorem]{Lemma}
\newtheorem{proposition}[theorem]{Proposition}
\newtheorem{observation}[theorem]{Observation}
\newtheorem{problem}[theorem]{Problem}
\newtheorem{example}[theorem]{Example}

\theoremstyle{definition}
\newtheorem{definition}[theorem]{Definition}
\theoremstyle{remark}
\newtheorem{remark}[theorem]{Remark}

\newcommand\remove[1]{}

\def\cay{\hskip0.02cm{\rm Cay}\hskip0.01cm}

\def\f2{\mathbb{F}_2}

\newcommand{\RNP}{{\rm RNP}}

\newcommand{\ep}{\varepsilon}

\newcommand{\diam}{{\rm diam}\hskip0.02cm}

\begin{document}

\title{\LARGE Metric characterizations  of some classes of Banach spaces}

\author{Mikhail Ostrovskii}

\date{\today}
\maketitle

\begin{large}

\hfill{\sl Dedicated to the memory of Cora Sadosky}\vskip1cm

\noindent{\bf Abstract.} The main purpose of the paper is to
present some recent results on metric characterizations of
superreflexivity and the Radon-Nikod\'ym property.
\medskip

\medskip

\medskip

\tableofcontents

\section{Introduction}

By a \emph{metric characterization} of a class of Banach spaces in
the most general sense we mean a characterization which refers
only to the metric structure of a Banach space and does not
involve the linear structure. Some origins of the idea of a metric
characterization can be seen in the classical theorem of Mazur and
Ulam \cite{MU32}: Two Banach spaces (over reals) are isometric as
metric spaces if and only if they are linearly isometric as Banach
spaces.

However study of metric characterizations became an active
research direction only in mid-1980s, in the work of Bourgain
\cite{Bou86} and Bourgain-Milman-Wolfson \cite{BMW86}. This study
was motivated by the following result of Ribe \cite{Rib76}.

\begin{definition} Let $X$ and $Y$ be two Banach spaces. The space $X$
is said to be \emph{finitely representable} \index{finitely
representable} in $Y$ if for any $\ep>0$ and any
finite-dimensional subspace $F\subset X$ there exists a
finite-dimensional subspace $G\subset Y$ such that $d(F,G)<1+\ep$,
where $d(F,G)$ is the Banach-Mazur distance.

The space $X$ is said to be \emph{crudely finitely representable}
\index{crudely finitely representable} in $Y$ if there exists
$1\le C<\infty$ such that for any finite-dimensional subspace
$F\subset X$ there exists a finite-dimensional subspace $G\subset
Y$ such that $d(F,G)\le C$.
\end{definition}

\begin{theorem}[Ribe \cite{Rib76}]\label{T:Ribe} Let $Z$ and $Y$ be Banach spaces. If
$Z$ and $Y$ are uniformly homeomorphic, then $Z$ and $Y$ are
crudely finitely representable in each other.
\end{theorem}

Three proofs of this theorem are known at the moment:

\begin{itemize}

\item The original proof of Ribe \cite{Rib76}. Some versions of it
were presented in Enflo's survey \cite{Enf76} and the book by
Benyamini and Lindenstrauss \cite[pp.~222--224]{BL00}.

\item The proof of Heinrich-Mankiewicz \cite{HM82} based on
ultraproduct techniques, also presented in \cite{BL00}.

\item The proof of Bourgain \cite{Bou87} containing related
quantitative estimates. This paper is a very difficult reading.
The proof has been clarified and simplified by
Giladi-Naor-Schechtman \cite{GNS12} (one of the steps was
simplified earlier by Begun \cite{Beg99}). The presentation of
\cite{GNS12} is easy to understand, but some of the $\ep$-$\delta$
ends in it do not meet. I tried to fix this when I presented this
result in my book \cite[Section 9.2]{Ost13a} (let me know if you
find any problems with $\ep$-$\delta$ choices there).
\end{itemize}

These three proofs develop a wide spectrum of methods of the
nonlinear Banach space theory and are well worth studying.
\medskip

The Ribe theorem implies stability under uniform homeomorphisms of
each class $\mathcal{P}$ of Banach spaces satisfying the following
condition LHI (local, hereditary, isomorphic): if
$X\in\mathcal{P}$ and $Y$ crudely finitely representable in $X$,
then $Y\in \mathcal{P}$.
\medskip

The following well-known classes have the described property:

\begin{itemize}

\item superreflexive spaces (see the definition and related
results in Section \ref{S:SuperR} of this paper),

\item spaces having cotype $q$, $q\in[2,\infty)$ (the definitions
of type and cotype can be found, for example, in \cite[Section
2.4]{Ost13a}),

\item spaces having cotype $r$ for each $r>q$ where
$q\in[2,\infty)$,

\item spaces having type $p$, $p\in(1,2]$,

\item spaces having type $r$ for each $r<p$ where $p\in(1,2]$;

\item Banach spaces isomorphic to $q$-convex spaces
$q\in[2,\infty)$ (see Definition \ref{D:ModConvQconv} below, more
details can be found, for example, in \cite[Section 8.4]{Ost13a}),

\item Banach spaces isomorphic to $p$-smooth spaces $p\in(1,2]$
(see Definition \ref{D:Psmooth} and \cite[Section 8.4]{Ost13a}),

\item UMD (unconditional for martingale differences) spaces
(recommended source for information on the UMD property is the
forthcoming book \cite{Pis14+}),

\item Intersections of some collections of classes described
above,

\item One of such intersections is the class of spaces isomorphic
to Hilbert spaces (by the Kwapie\'n theorem \cite{Kwa72}, each
Banach space having both type $2$ and cotype $2$ is isomorphic to
a Hilbert space),

\item Banach spaces isomorphic to subspaces of the space
$L_p(\Omega,\Sigma,\mu)$ for some measure space
$(\Omega,\Sigma,\mu)$, $p\ne 2,\infty$ (for $p=\infty$ we get the
class of all Banach spaces, for $p=2$ we get the class of spaces
isomorphic to Hilbert spaces).

\end{itemize}

\begin{remark} This list seem to constitute
the list of all classes of Banach spaces satisfying the condition
LHI which were systematically studied.
\end{remark}

By the Ribe Theorem (Theorem \ref{T:Ribe}), one can expect that
each class satisfying the condition LHI has a metric
characterization. At this point metric characterizations are known
for all classes listed above except $p$-smooth and UMD (and some
of the intersections involving these classes). Here are the
references:

\begin{itemize}

\item Superreflexivity - see Section \ref{S:SuperR} of this paper
for a detailed account.

\item Properties related to type - \cite{MN07} (see \cite{Enf78},
\cite{BMW86}, and \cite{Pis86} for previous important results in
this direction, and \cite{GN10} for some improvements).

\item Properties related to cotype - \cite{MN08}, see \cite{GMN11}
for some improvements.

\item $q$-convexity - \cite{MN13}.

\item Spaces isomorphic to subspaces of $L_p$ $(p\ne 2,\infty)$.
Rabinovich noticed that one can generalize results of \cite{LLR95}
and characterize the optimal distortion of embeddings of a finite
metric space into $L_p$-space (see \cite[Exercise 4 on p.~383 and
comment of p.~380]{Mat02} and a detailed presentation in
\cite[Section 4.3]{Ost13a}). Johnson, Mendel, and Schechtman
(unpublished) found another characterization of the optimal
distortion using a modification of the argument of Lindenstrauss
and Pe\l czy\'nski \cite[Theorem 7.3]{LP68}. These
characterizations are very close to each other. They are not
satisfactory in some respects.

\end{itemize}

\subsection{Ribe program}\label{S:Ribe}

It should be mentioned that some of the metric characterizations
(for example of the class of spaces having some type $>1$)  can be
derived from the known `linear' theory. Substantially nonlinear
characterizations started with the paper of Bourgain \cite{Bou86}
in which he characterized superreflexive Banach spaces in terms of
binary trees.\medskip

This paper of Bourgain and the whole direction of metric
characterizations was inspired by the unpublished paper of Joram
Lindenstrauss with the tentative title ``Topics in the geometry of
metric spaces''. This paper has never been published (and
apparently has never been written, so it looks like it was just a
{\it conversation}, and not a {\it paper}), but it had a
significant impact on this direction of research. The unpublished
paper of Lindenstrauss and the mentioned paper of Bourgain
\cite{Bou86} initiated what is now known as the {\it Ribe
program}.

Bourgain \cite[p.~222]{Bou86} formulated it as the program of
search for equivalent definitions of different LHI invariants in
terms of metric structure with the next step consisting in
studying these metrical concepts in general metric spaces in an
attempt to develop an analogue of the linear theory.

Bourgain himself made several important contributions to the Ribe
program, now it is a very deep and extensive research direction.
In words of Ball \cite{Bal13}: ``Within a decade or two the Ribe
programme acquired an importance that would have been hard to
predict at the outset''. In this paper I am going to cover only a
very small part of known results on this program. I refer
interested people to the surveys of Ball \cite{Bal13} (short
survey) and Naor \cite{Nao12} (extensive survey).\medskip

Many of the known metric characterizations use the following
standard definitions.

\begin{definition}\label{D:lip} Let $0\le C<\infty$. A map $f: (A,d_A)\to (Y,d_Y)$ between two
metric spaces is called $C$-{\it Lipschitz} if
\[\forall u,v\in A\quad d_Y(f(u),f(v))\le Cd_A(u,v).\] A map $f$
is called {\it Lipschitz} if it is $C$-Lipschitz for some $0\le
C<\infty$.

Let $1\le C<\infty$. A map $f:A\to Y$ is called a {\it
$C$-bilipschitz embedding} if there exists $r>0$ such that
\begin{equation}\label{E:MapDist}\forall u,v\in A\quad rd_A(u,v)\le
d_Y(f(u),f(v))\le rCd_A(u,v).\end{equation} A {\it bilipschitz
embedding} is an embedding which is $C$-bilipschitz for some $1\le
C<\infty$. The smallest constant $C$ for which there exist $r>0$
such that \eqref{E:MapDist} is satisfied is called the {\it
distortion} of $f$.
\end{definition}

\remove{We shall restrict our attention to classes $\mathcal{P}$
of Banach spaces which are \emph{hereditary} and \emph{isomorphic
invariant}. {\it Hereditary}  means that if $X\in\mathcal{P}$ and
$Y$ is a closed subspace of $X$, then $Y\in \mathcal{P}$. {\it
Isomorphic invariant} means that $X\in\mathcal{P}$ implies that
all Banach spaces isomorphic to $X$ are also in $\mathcal{P}$. The
reason is that it is difficult to see any perspectives for metric
characterizations of types which we are going to consider for
classes which do not satisfy these conditions.
\medskip}

There are at least two directions in which we can seek metric
characterizations:

\begin{itemize}

\item[{\bf (1)}] We can try to characterize metric spaces which
admit bilipschitz embeddings into some Banach spaces belonging to
$\mathcal{P}$.

\item[{\bf (2)}] We can try to find metric structures which are
present in each Banach space $X\notin \mathcal{P}$.

\end{itemize}

Characterizations of type {\bf (1)} would be much more interesting
for applications. However, as far as I know such characterizations
were found only in the following cases: (i) $\mathcal{P}=\{$the
class of Banach spaces isomorphic to a Hilbert space$\}$ (it is
the Linial-London-Rabinovich \cite[Corollary 3.5]{LLR95} formula
for distortion of  embeddings of a finite metric space into
$\ell_2$). (ii) $\mathcal{P}=\{$the class of Banach spaces
isomorphic to a subspace of some $L_p$-space$\}$, $p$ is a fixed
number $p\ne2,\infty$, see the last paragraph preceding Section
\ref{S:Ribe}.

\subsection{Local properties for which no metric characterization is known}

\begin{problem}\label{P:UMD} Find a metric characterization of UMD.
\end{problem}

Here {\it UMD} stays for {\it unconditional for martingale
differences}. The most comprehensive source of information on UMD
is the forthcoming book of Pisier \cite{Pis14+}.
\smallskip

I have not found in the literature any traces of attempts to work
on Problem \ref{P:UMD}.
\bigskip

\begin{definition}\label{D:Psmooth} A Banach space is called {\it $p$-smooth} if its modulus of smoothness
satisfies $\rho(t)\le Ct^p$ for $p\in(1,2]$.
\end{definition}

See \cite[Section 8.4]{Ost13a} for information on $p$-smooth
spaces.

\begin{problem}\label{P:p-smooth} Find a metric characterization of the class of
Banach spaces isomorphic to $p$-smooth spaces $p\in(1,2]$.
\end{problem}

This problem was posed and discussed in the paper by Mendel and
Naor \cite{MN13}, where a similar problem is solved for $q$-convex
spaces. Mendel and Naor wrote \cite[p.~335]{MN13}: ``Trees are
natural candidates for finite metric obstructions to
$q$-convexity, but it is unclear what would be the possible finite
metric witnesses to the ``non-$p$-smoothness'' of a metric
space''.

\section{Metric characterizations of
superreflexivity}\label{S:SuperR}

\begin{definition}[James \cite{Jam72a,Jam72b}]\label{D:SuperRefl} A Banach space $X$ is called \emph{super\-re\-fle\-xive} if each Banach space which is finitely representable in $X$
is reflexive.
\end{definition}

It might look like a rather peculiar definition, but, as I
understand, introducing it ($\approx 1967$) James  already had a
feeling that it is a very natural and important definition. This
feeling was shown to be completely justified when Enflo
\cite{Enf72} completed the series of results of James by proving
that each superreflexive space has an equivalent uniformly convex
norm.

\begin{definition}\label{D:UC} {\rm A Banach space is called
\emph{uniformly convex}  if for every $\ep > 0$ there is some
$\delta > 0$ so that for any two vectors with $\|x\|\le1$ and
$\|y\|\le 1$, the inequality
\[ 1-\left\|\frac{x+y}2\right\|<\delta\] implies
\[   \|x-y\|<\varepsilon.\]}
\end{definition}

\begin{definition}\label{D:EquivNorm} Two norms $||\cdot||_1$ and
$||\cdot||_2$ on a linear space $X$ are called \emph{equivalent}
if there are constants $0<c\le C<\infty$ such that \[c||x||_1\le
||x||_2\le C||x||_1\] for each $x\in X$.
\end{definition}

After the pioneering results of James and Enflo numerous
equivalent reformulations of superreflexivity were found and
superreflexivity was used in many different contexts.

The metric characterizations of superreflexivity which we are
going to present belong to the class of so-called \emph{test-space
characterizations}.
\medskip

\begin{definition} Let $\mathcal{P}$ be a class of Banach
spaces and let $T=\{T_\alpha\}_{\alpha\in A}$ be a set of metric
spaces. We say that $T$ is a set of {\it test-spaces} for
$\mathcal{P}$ if the following two conditions are equivalent: {\bf
(1)} $X\notin\mathcal{P}$; {\bf (2)} The spaces
$\{T_\alpha\}_{\alpha\in A}$ admit bilipschitz embeddings into $X$
with uniformly bounded distortions.
\end{definition}

\subsection{Characterization of superreflexivity in terms of binary
trees}

\begin{definition}\label{D:trees} A {\it binary tree of  depth $n$}, denoted $T_n$, is a
finite graph in which each vertex is represented by a finite
(possibly empty) sequence of $0$ and $1$, of length at most $n$.
Two vertices in $T_n$ are adjacent if the sequence corresponding
to one of them is obtained from the sequence corresponding to the
other by adding one term on the right. (For example, vertices
corresponding to $(1,1,1,0)$ and $(1,1,1,0,1)$ are adjacent.) A
vertex corresponding to a sequence of length $n$ in $T_n$ is
called a {\it leaf}.

An {\it infinite binary tree}, denoted $T_\infty$, is an infinite
graph in which each vertex is represented by a finite (possibly
empty) sequence of $0$ and $1$. Two vertices in $T_\infty$ are
adjacent if the sequence corresponding to one of them is obtained
from the sequence corresponding to the other by adding one term on
the right.

Both for finite and infinite binary trees we use the following
terminology. The vertex corresponding to the empty sequence is
called a {\it root}. If a sequence $\tau$ is an initial segment of
the sequence $\sigma$ we say that $\sigma$ is a {\it descendant}
of $\tau$ and that $\tau$ is an {\it ancestor} of $\sigma$. If a
descendant $\sigma$ of $\tau$ is adjacent to $\tau$, we say that
$\sigma$ is a {\it child} of $\tau$ and that $\tau$ is a {\it
parent} of $\sigma$. Two children of the same parent are called
{\it siblings}. Child of a child is called a {\it grandchild}. (It
is clear that each vertex in $T_\infty$ has exactly two children,
the same is true for all vertices of $T_n$ except leaves.)
\end{definition}

\begin{theorem}[Bourgain \cite{Bou86}]\label{T:Bourgain} A Banach space $X$ is
nonsuperreflexive if and only if it admits bilipschitz embeddings
with uniformly bounded distortions of finite binary trees
$\{T_n\}_{n=1}^\infty$ of all depths.
\end{theorem}

{

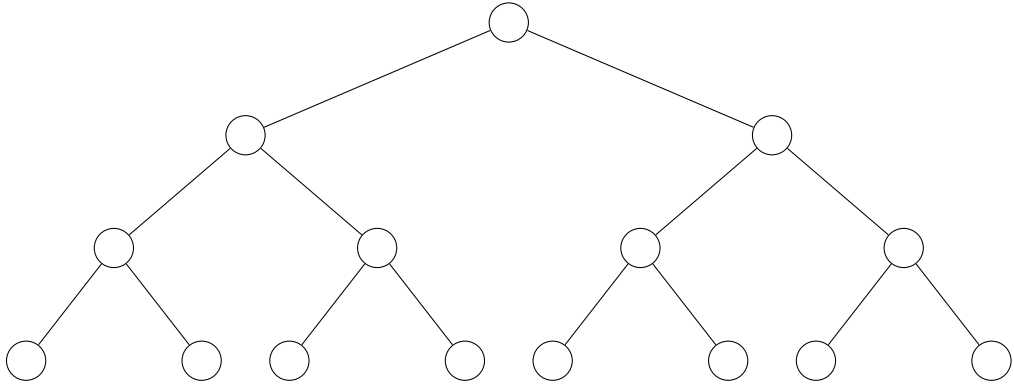
\begin{figure}
{
\begin{tikzpicture}[level/.style={sibling distance=70mm/#1}]
\node [circle,draw] (z) {\hbox{~~}}
  child {node [circle,draw] (a) {\hbox{~~}}
    child {node [circle,draw] (b) {\hbox{~~}}
      child {node [circle,draw]  {\hbox{~~}}}
      child {node [circle,draw]  {\hbox{~~}}}
    }
    child {node [circle,draw] (g) {\hbox{~~}}
      child {node [circle,draw]  {\hbox{~~}}}
      child {node [circle,draw]  {\hbox{~~}}}
    }
  }
  child {node [circle,draw] (a) {\hbox{~~}}
    child {node [circle,draw] (b) {\hbox{~~}}
      child {node [circle,draw]  {\hbox{~~}}}
      child {node [circle,draw]  {\hbox{~~}}}
    }
    child {node [circle,draw] (g) {\hbox{~~}}
      child {node [circle,draw]  {\hbox{~~}}}
      child {node [circle,draw]  {\hbox{~~}}}
      }
    };
\end{tikzpicture}
} \caption{The binary tree of depth $3$, that is, $T_3$.}
\label{F:Tree}
\end{figure}
}

In Bourgain's proof the  difficult direction is the ``if''
direction, the ``only if'' is an easy consequence of the theory of
superreflexive spaces. Recently Kloeckner \cite{Klo14} found a
very simple proof of the ``if'' direction. I plan to describe the
proofs of Bourgain and Kloeckner after recalling the results on
superreflexivity which we need.

\begin{definition}\label{D:ModConvQconv} The {\it modulus of } (uniform) {\it convexity} $\delta_X(\varepsilon)$ of
a Banach space $X$ with norm $\|\cdot\|$ is defined as
\[\inf\left\{1-\left\|\frac{x+y}2\right\| \,\middle|\,
  \| x\| = \| y\| =1 \ \textrm{and}\ \| x-y\|\geqslant \varepsilon\right\}\]
for $\varepsilon\in (0,2]$. The space $X$ or its norm is said to
be {\it $q$-convex}, $q\in[2,\infty)$ if
$\delta_X(\varepsilon)\geqslant c \varepsilon^q$ for some $c>0$.
\end{definition}

\begin{remark} It is easy to see that the definition of the
uniform convexity given in Definition \ref{D:UC} is equivalent to:
$X$ is {\it uniformly convex} if and only if
$\delta_X(\varepsilon)>0$ for each $\varepsilon\in (0,2]$.
\end{remark}

\begin{theorem}[Pisier \cite{Pis75}]\label{T:Pisier} The following properties of a Banach space $Y$ are
equivalent:

\begin{enumerate}

\item\label{I:SR} $Y$ is superreflexive.

\item\label{I:PConvex}  $Y$ has an equivalent $q$-convex norm for
some $q\in[2,\infty)$. \medskip
\end{enumerate}
\end{theorem}

Using Theorem \ref{T:Pisier} we can prove the ``if'' part of
Bourgain's characterization. Denote by $c_X(T_n)$ the infimum of
distortions of embeddings of the binary tree $T_n$ into a Banach
space $X$.

By Theorem \ref{T:Pisier}, for the ``if'' part of Bourgain's
theorem it suffices to prove that if $X$ is $q$-convex, then for
some $c_1>0$ we have
\[c_X(T_n)\ge c_1(\log_2n)^{\frac1q}.\]

\begin{proof}[Proof {\rm (Kloeckner \cite{Klo14})}]
Let $F$ be the four-vertex tree with one root $a_0$ which has one
child $a_1$ and two grand-children $a_2$, $a'_2$. Sometimes such
tree is called a {\it fork}, see Figure \ref{F:Fork}. The
following lemma is similar to the corresponding results in
\cite{Mat99}.

{

\begin{figure}
\begin{center}
{
\begin{tikzpicture}[level/.style={sibling distance=70mm/#1}]
\node [circle,draw] (z) {$a_0$}
  child {node [circle,draw] (a) {$a_1$}
         child {node [circle,draw]  {$a_2$}}
      child {node [circle,draw]  {$a_2'$}}
    };
\end{tikzpicture}
} \caption{A fork.} \label{F:Fork}
\end{center}
\end{figure}

}

\begin{lemma}\label{L:GrandChi}
There is a constant $K=K(X)$ such that if $\varphi:F\to X$ is
$D$-Lipschitz and distance non-decreasing, then either
\[\|\varphi(a_0)-\varphi(a_2)\|\leqslant 2\left(D-\frac{K}{D^{q-1}}\right)\]
or
\[\|\varphi(a_0)-\varphi(a'_2)\|\leqslant 2\left(D-\frac{K}{D^{q-1}}\right)\]
\end{lemma}

First we finish the proof of $c_X(T_n)\ge c_1(\log_2n)^{\frac1q}$
using Lemma \ref{L:GrandChi}. So let $\varphi:T_n\to X$ be a map
of distortion $D$. Since $X$ is a Banach space, we may assume that
$\varphi$ is $D$-Lipschitz, distance non-decreasing map, that is
\[d_{T_n}(u,v)\le ||\varphi(u)-\varphi(v)||\le Dd_{T_n}(u,v).\] The
main idea of the proof is to construct a less-distorted embedding
of a smaller tree.

Given any vertex of $T_n$ which is not a leaf, let us name
arbitrarily one of its two children its {\it daughter}, and the
other its {\it son}. We select two grandchildren of the root in
the following way: we pick the grandchild mapped by $\varphi$
closest to the root among its daughter's children and the
grandchild mapped by $\varphi$ closest to the root among its son's
children (ties are resolved arbitrarily). Then we select
inductively, in the same way, two grandchildren for all previously
selected vertices up to generation $n-2$.
\medskip

The set of selected vertices, endowed with half the distance
induced by the tree metric is isometric to $T_{\lfloor\frac
n2\rfloor}$, and Lemma \ref{L:GrandChi} implies that the
restriction of $\varphi$ to this set has distortion at most
$f(D)=D-\frac{K}{D^{q-1}}$.

We can iterate such restrictions $\lfloor \log_2n\rfloor$ times to
get an embedding of $T_1$ whose distortion is at most
\[D-\lfloor \log_2n\rfloor\frac{K}{D^{q-1}}\]
since each iteration improves the distortion by at least
$K/D^{q-1}$. Since the distortion of any embedding is at least
$1$, we get the desired inequality.
\end{proof}

\begin{remark} Kloeckner borrowed the approach based on `controlled improvement for embeddings of smaller parts' from the Johnson-Schechtman
paper \cite{JS09} in which it is used for diamond graphs
(Kloeckner calls this approach a {\it self-improvement argument}).
Arguments of this type are well-known and widely used in the
linear theory, where they go back at least to James \cite{Jam64a};
but these two examples (Johnson-Schechtman \cite{JS09} and
Kloeckner \cite{Klo14}) seem to be the only two known results of
this type in the nonlinear theory. It would be interesting to find
further nonlinear arguments of this type.
\end{remark}

\begin{proof}[Sketch of the proof of Lemma \ref{L:GrandChi}]
Assume $\varphi(a_0)=0$ and let $x_1=\varphi(a_1)$,
$x_2=\varphi(a_2)$ and
 $x'_2=\varphi(a'_2)$. Recall that we assumed
 \[d_{T_n}(u,v)\le ||\varphi(u)-\varphi(v)||\le Dd_{T_n}(u,v). \eqno{(*)}\]

Consider the (easy) case where $\|x_2\|= 2D$ and $\|x'_2\|= 2D$
(that is, the distortion $D$ is attained on these vectors). We
claim that this implies that $x_2=x_2'$. In fact, it is easy to
check that this implies $\|x_1\|=D$, $\|x_2-x_1\|=D$, and
$\|x'_2-x_1\|=D$. Also $\left\|\frac{x_1+(x_2-x_1)}2\right\|=D$
and $\left\|\frac{x_1+(x'_2-x_1)}2\right\|=D$. By the uniform
convexity we get $||x_1-(x_2-x_1)||=0$ and $||x_1-(x'_2-x_1)||=0$.
Hence $x_2=x_2'$, and we get that the conditions $\|x_2\|= 2D$ and
$\|x'_2\|= 2D$ cannot be satisfied simultaneously.

The proof of Lemma \ref{L:GrandChi}  goes as follows. We start by
letting $\|x_2\|\ge 2(D-\eta)$ and $\|x'_2\|\ge 2(D-\eta)$ for
some $\eta>0$. Using a perturbed version of the argument just
presented, the definition of the modulus of convexity, and our
assumption $\delta_X(\ep)\ge c\ep^p$, we get an estimate of
$||x_2-x'_2||$ from above in terms of $\eta$. Comparing this
estimate with the assumption $||x_2-x_2'||\ge 2$ (which follows
from $d_{T_n}(u,v)\le ||\varphi(u)-\varphi(v)||$), we get the
desired estimate for $\eta$ from below, see \cite{Klo14} for
details.
\end{proof}

\begin{remark} The approach of Kloeckner can be used for any uniformly convex space, it is not necessary to combine it with the Pisier Theorem
(Theorem \ref{T:Pisier}), see \cite{Pis14+}.
\end{remark}

To prove the ``only if'' part of Bourgain's Theorem we need the
following characterization of superreflexivity,  one of the most
suitable sources for this characterization of superreflexivity is
\cite{Pis14+}.

\begin{theorem}[James \cite{Jam64a,Jam72a,SS70}]\label{T:NSRChar}
Let $X$ be a Banach space. The following are equivalent:
\smallskip

\begin{enumerate}

\item\label{I:NonSR} $X$ is not superreflexive

\item\label{I:Flat} There exists $\alpha\in(0,1]$ such that for
each $m\in\mathbb{N}$ the unit ball of the space $X$ contains a
finite sequence $x_1, x_2,\dots,x_m$ of vectors satisfying, for
any $j\in\{1,\dots,m-1\}$ and any real coefficients
$a_1,\dots,a_m$, the condition
\begin{equation}\label{E:CharSR}
\left\|\sum_{i=1}^ma_ix_i\right\|\ge
\alpha\,\left(\left|\sum_{i=1}^ja_i\right|+\left|\sum_{i=j+1}^ma_i\right|
\right).
\end{equation}

\item\label{I:FlatForall} For each $\alpha\in(0,1)$ and each
$m\in\mathbb{N}$ the unit ball of the space $X$ contains a finite
sequence $x_1, x_2,\dots,x_m$ of vectors satisfying, for any
$j\in\{1,\dots,m-1\}$ and any real coefficients $a_1,\dots,a_m$,
the condition \eqref{E:CharSR}.

\end{enumerate}
\end{theorem}

\begin{remark} It is worth mentioning that the proof of \eqref{I:NonSR}$\Rightarrow$\eqref{I:FlatForall} in the case
where $\alpha\in[\frac12,1)$ is much more difficult than in the
case $\alpha\in(0,\frac12)$. A relatively easy proof in the case
$\alpha\in[\frac12,1)$ was found by Brunel and Sucheston
\cite{BS75}, see also its presentation in \cite{Pis14+}.
\end{remark}

\begin{remark}
To prove the Bourgain theorem it suffices to use
\eqref{I:NonSR}$\Rightarrow$\eqref{I:FlatForall} in the `easy'
case $\alpha\in(0,\frac12)$. The case $\alpha\in[\frac12,1)$ is
needed only for ``almost-isometric'' embeddings of trees into
nonsuperreflexive spaces.
\end{remark}

\begin{remark}
The equivalence of
\eqref{I:Flat}$\Leftrightarrow$\eqref{I:FlatForall} in Theorem
\ref{T:NSRChar} can be proved using a ``self-improvement
argument'', but the proof of James is different. A proof of
\eqref{I:Flat}$\Leftrightarrow$\eqref{I:FlatForall} using a
``self-improvement argument'' was obtained by Wenzel \cite{Wen97},
it is based on the Ramsey theorem, so it requires a very lengthy
sequence to get a better $\alpha$. In \cite{Ost04} it was proved
that to some extent the usage of `very lengthy' sequences is
necessary.
\end{remark}

\begin{proof}[Proof of the ``only if'' part]
There is a natural partial order on $T_n$: we say that $s < t$
$(s,t\in T_n)$ if the sequence corresponding to $s$ is the initial
segment of the sequence corresponding to $t$.

{

\begin{figure}
{
\begin{tikzpicture}[level/.style={sibling distance=70mm/#1}]
\node [circle,draw] (z) {0}
  child {node [circle,draw] (a) {$-\frac12$}
    child {node [circle,draw] (b) {$-\frac34$}
      child {node [circle,draw]  {$-\frac78$}}
      child {node [circle,draw]  {$-\frac{5}{8}$}}
    }
    child {node [circle,draw] (g) {$-\frac14$}
      child {node [circle,draw]  {$-\frac38$}}
      child {node [circle,draw]  {$-\frac18$}}
    }
  }
  child {node [circle,draw] (a) {$\frac12$}
    child {node [circle,draw] (b) {$\frac14$}
      child {node [circle,draw]  {$\frac18$}}
      child {node [circle,draw]  {$\frac38$}}
    }
    child {node [circle,draw] (g) {$\frac34$}
      child {node [circle,draw]  {$\frac58$}}
      child {node [circle,draw]  {$\frac78$}}
      }
    };
\end{tikzpicture}
} \caption{The map of $T_3$ into $[-1,1]$.} \label{F:TreeMap}
\end{figure}
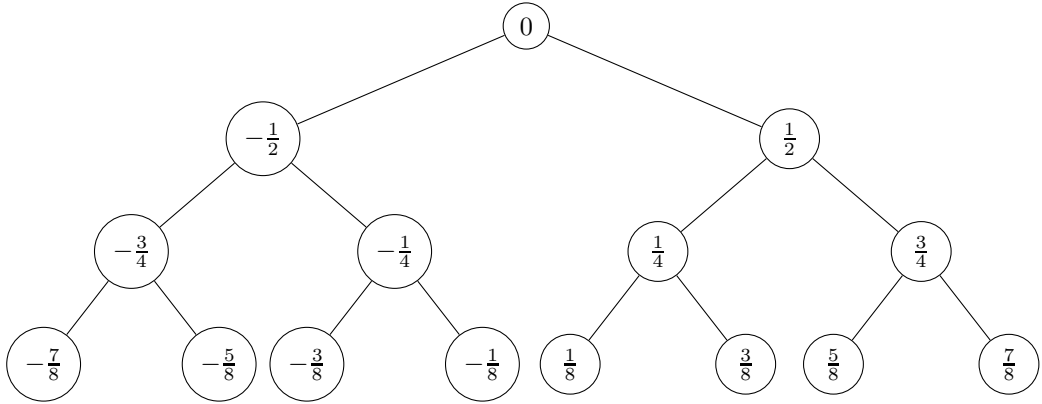

}

An important observation of Bourgain is that there is a bijective
mapping
\[\varphi: T_n\to[1,\dots,2^{n+1}-1]\] such that $\varphi$ maps
two disjoint intervals of the ordering of $T_n$, starting at the
same vertex and going `down' into disjoint intervals of
$[1,\dots,2^{n+1}-1]$. The existence of $\varphi$ can be seen from
a suitably drawn picture of $T_n$ (see Figure \ref{F:TreeMap}), or
using the expansion of numbers in base $2$. To use the expansion
of numbers, we observe that the map
$\{\theta_i\}_{i=1}^n\to\left\{{2\theta_i-1}\right\}_{i=1}^n$ maps
a $0,1-$sequence onto the corresponding $\pm1-$sequence. Now we
introduce a map $\psi:T_n\to [-1,1]$ by letting
$\psi(\emptyset)=0$ and
\[\psi(\theta_1,\dots,\theta_n)=\sum_{i=1}^n2^{-i}\left({2\theta_i-1}\right).\]
To construct $\varphi$ we relabel the range of $\psi$ in the
increasing order using numbers $[1,\dots,2^{n+1}-1]$.

Let $\{x_1, x_2,\dots,x_{2^{n+1}-1}\}$ be a sequence in a
nonsuperreflexive Banach space $X$ whose existence is guaranteed
by Theorem \ref{T:NSRChar}
(\eqref{I:NonSR}$\Rightarrow$\eqref{I:FlatForall}). We introduce
an embedding $F_n : T_n \to X$ by
\[F_n(t) =\sum_{s\le t} x_{\varphi(s)},\]
where $s\le t$ for vertices of a binary tree means that $s$ is the
initial segment of the sequence $t$. Then
$F_n(t_1)-F_n(t_2)=\sum_{t_0< s\le t_1} x_{\varphi(s)}-\sum_{t_0<
s\le t_2} x_{\varphi(s)}$, where $t_0$ is the vertex of $T_n$
corresponding to the largest initial common segment of $t_1$ and
$t_2$, see Figure \ref{F:ComAns}. The condition in
\eqref{E:CharSR} and the choice of $\varphi$ imply that
\[||F_n(t_1)-F_n(t_2)||\ge\alpha(d_T(t_1,t_0)+d_T
(t_2,t_0))=\alpha d_{T_n}(t_1,t_2).\] The estimate
$||F_n(t_1)-F_n(t_2)||$ from above is straightforward. This
completes the proof of bilipschitz embeddability of $\{T_n\}$ into
any nonsuperreflexive Banach space with uniformly bounded
distortions.\end{proof}

{

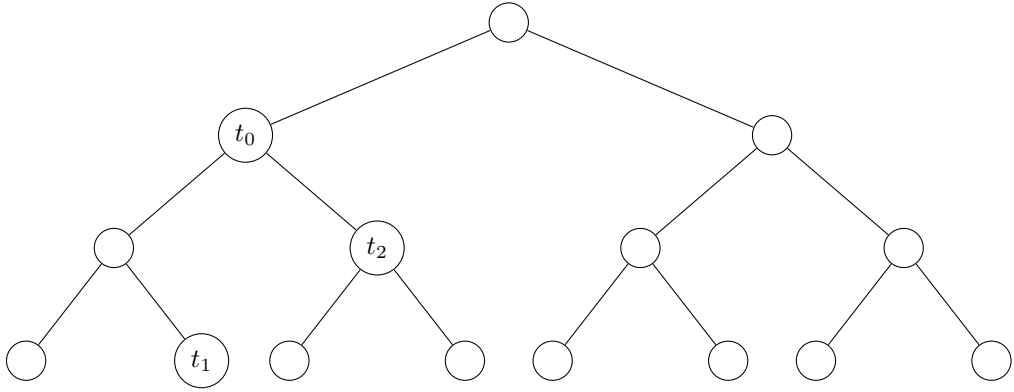
\begin{figure}
{
\begin{tikzpicture}[level/.style={sibling distance=70mm/#1}]
\node [circle,draw] (z) {\hbox{~~}}
  child {node [circle,draw] (a) {$t_0$}
    child {node [circle,draw] (b) {\hbox{~~}}
      child {node [circle,draw]  {\hbox{~~}}}
      child {node [circle,draw]  {$t_1$}}
    }
    child {node [circle,draw] (g) {$t_2$}
      child {node [circle,draw]  {\hbox{~~}}}
      child {node [circle,draw]  {\hbox{~~}}}
    }
  }
  child {node [circle,draw] (a) {\hbox{~~}}
    child {node [circle,draw] (b) {\hbox{~~}}
      child {node [circle,draw]  {\hbox{~~}}}
      child {node [circle,draw]  {\hbox{~~}}}
    }
    child {node [circle,draw] (g) {\hbox{~~}}
      child {node [circle,draw]  {\hbox{~~}}}
      child {node [circle,draw]  {\hbox{~~}}}
      }
    };
\end{tikzpicture}
} \caption{$t_0$ is the closest common ancestor of $t_1$ and
$t_2$.} \label{F:ComAns}
\end{figure}
}

\subsection{Characterization of superreflexivity in terms of
diamond graphs}

Johnson and Schechtman \cite{JS09} proved that there are some
other sequences of graphs (with their graph metrics) which also
can serve as test-spaces for superreflexivity. For example, binary
trees in Bourgain's Theorem can be replaced by the diamond graphs
or by Laakso graphs.

\begin{definition}\label{D:Diamonds}
Diamond graphs $\{D_n\}_{n=0}^\infty$ are defined as follows: The
{\it diamond graph} of level $0$ is denoted $D_0$. It has two
vertices joined by an edge of length $1$. The {\it diamond graph}
$D_n$ is obtained from $D_{n-1}$ as follows. Given an edge $uv\in
E(D_{n-1})$, it is replaced by a quadrilateral $u, a, v, b$, with
edges $ua$, $av$, $vb$, $bu$. (See Figure \ref{F:Diamond2}.)
\medskip

Two different normalizations of the graphs $\{D_n\}_{n=1}^\infty$
are considered:

\begin{itemize}

\item {\it Unweighted diamonds:} Each edge has length $1$.

\item {\it Weighted diamonds:} Each edge of $D_n$ has length
$2^{-n}$
\end{itemize}

In both cases we endow vertex sets of $\{D_n\}_{n=0}^\infty$ with
their shortest path metrics.

In the case of weighted diamonds the identity map $D_{n-1}\mapsto
D_n$ is an isometry. In this case the union of $D_n$, endowed with
its the metric induced from $\{D_n\}_{n=0}^\infty$ is called the
{\it infinite diamond} and is denoted $D_\omega$.
\end{definition}

To the best of my knowledge the first paper in which diamond
graphs $\{D_n\}_{n=0}^\infty$ were used in Metric Geometry is
\cite{GNRS04} (a conference version was published in 1999).

{

\begin{figure}
\begin{center}
{
\begin{tikzpicture}
  [scale=.25,auto=left,every node/.style={circle,draw}]
  \node (n1) at (16,0) {\hbox{~~~}};
  \node (n2) at (5,5)  {\hbox{~~~}};
  \node (n3) at (11,11)  {\hbox{~~~}};
  \node (n4) at (0,16) {\hbox{~~~}};
  \node (n5) at (5,27)  {\hbox{~~~}};
  \node (n6) at (11,21)  {\hbox{~~~}};
  \node (n7) at (16,32) {\hbox{~~~}};
  \node (n8) at (21,21)  {\hbox{~~~}};
  \node (n9) at (27,27)  {\hbox{~~~}};
  \node (n10) at (32,16) {\hbox{~~~}};
  \node (n11) at (21,11)  {\hbox{~~~}};
  \node (n12) at (27,5)  {\hbox{~~~}};

  \foreach \from/\to in {n1/n2,n1/n3,n2/n4,n3/n4,n4/n5,n4/n6,n6/n7,n5/n7,n7/n8,n7/n9,n8/n10,n9/n10,n10/n11,n10/n12,n11/n1,n12/n1}
    \draw (\from) -- (\to);

\end{tikzpicture}
}
\caption{Diamond $D_2$.}\label{F:Diamond2}
\end{center}
\end{figure}
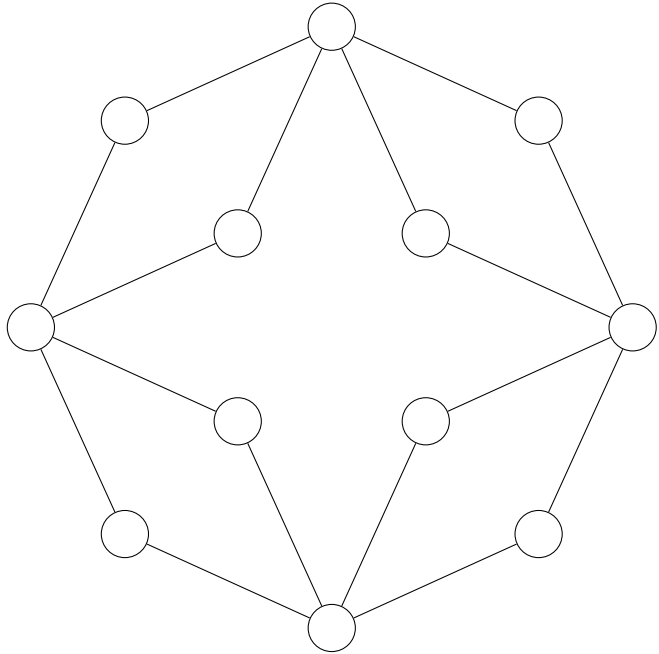
}

\begin{definition}\label{D:Laakso}
Laakso graphs $\{L_n\}_{n=0}^\infty$ are defined as follows: The
{\it Laakso graph} of level $0$ is denoted $L_0$. It has two
vertices joined by an edge of length $1$. The {\it Laakso graph}
$L_n$ is obtained from $L_{n-1}$ as follows. Given an edge $uv\in
E(L_{n-1})$, it is replaced by the graph $L_1$ shown in Figure
\ref{F:Laakso}, the vertices $u$ and $v$ are identified with the
vertices of degree $1$ of $L_1$.
\medskip

Two different normalizations of the graphs $\{L_n\}_{n=1}^\infty$
are considered:

\begin{itemize}

\item {\it Unweighted Laakso graphs:} Each edge has length $1$.

\item {\it Weighted Laakso graphs:} Each edge of $L_n$ has length
$4^{-n}$
\end{itemize}

In both cases we endow vertex sets of $\{L_n\}_{n=0}^\infty$ with
their shortest path metrics.

In the case of weighted Laakso graphs the identity map
$L_{n-1}\mapsto L_n$ is an isometry. In this case the union of
$L_n$, endowed with its the metric induced from
$\{L_n\}_{n=0}^\infty$ is called the {\it Laakso space} and is
denoted $L_\omega$.
\end{definition}

The Laakso graphs were introduced in \cite{LP01}, but they were
inspired by the construction of Laakso in \cite{Laa00}.

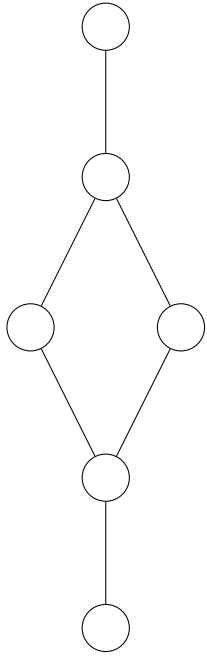
\begin{figure}
\begin{center}
{
\begin{tikzpicture}
  [scale=.25,auto=left,every node/.style={circle,draw}]
  \node (n1) at (16,0) {\hbox{~~~}};
  \node (n2) at (16,8)  {\hbox{~~~}};
  \node (n3) at (12,16) {\hbox{~~~}};
  \node (n4) at (20,16)  {\hbox{~~~}};
  \node (n5) at (16,24)  {\hbox{~~~}};
  \node (n6) at (16,32) {\hbox{~~~}};

\foreach \from/\to in {n1/n2,n2/n3,n2/n4,n4/n5,n3/n5,n5/n6}
    \draw (\from) -- (\to);

\end{tikzpicture}
} \caption{Laakso graph $L_1$.}\label{F:Laakso}
\end{center}
\end{figure}

\begin{theorem}[Johnson-Schechtman \cite{JS09}]\label{T:JS}
A Banach space $X$ is nonsuper\-re\-fle\-xive if and only if it
admits bilipschitz embeddings with uniformly bounded distortions
of diamonds $\{D_n\}_{n=1}^\infty$ of all sizes.
\end{theorem}

\begin{theorem}[Johnson-Schechtman \cite{JS09}]\label{T:JSLaakso}
A similar result holds for  $\{L_n\}_{n=1}^\infty$.
\end{theorem}

\begin{proof}[Without proof] These results, whose original proofs (especially for diamond graphs) are elegant in both directions, are loved by expositors.
Proof of Theorem \ref{T:JS} is presented in the lecture notes of
Lancien \cite{Lan13}, in the book of Pisier \cite{Pis14+}, and in
my book \cite{Ost13a}.
\end{proof}

\begin{remark} In the ``if'' direction of Theorem \ref{T:JS}, in addition to the original (controlled improvement for embeddings of smaller parts)
argument of Johnson-Schechtman \cite{JS09}, there are two other
arguments:

(1) The argument based on Markov convexity (see Definition
\ref{D:MarkConv}). It is obtained by combining results of
Lee-Naor-Peres \cite{LNP09} (each superreflexive Banach space is
Markov $p$-convex for some $p\in [2,\infty)$) and Mendel-Naor
\cite{MN13} (Markov convexity constants of diamond graphs are not
uniformly bounded from below, actually in \cite{MN13} this
statement is proved for Laakso graphs, but similar argument works
for diamond graphs).

(2) The argument of \cite[Section 3.1]{Ost11} showing that
bilipschitz embeddability of diamond graphs with uniformly bounded
distortions implies the finite tree property of the space, defined
as follows:
\end{remark}

\begin{definition}[James \cite{Jam72a}]\label{D:DelTrees} Let $\delta>0$. A {\it $\delta$-tree} in a Banach space $X$ is a
subset $\{x_\tau\}_{\tau\in T_\infty}$ of $X$ labelled by elements
of the infinite binary tree $T_\infty$, such that for each
$\tau\in T_\infty$ we have
\begin{equation}\label{E:TreeConditions}
x_\tau=\frac12(x_{\sigma_1}+x_{\sigma_2})\quad\hbox{ and
}||x_\tau-x_{\sigma_1}||=||x_{\tau}-x_{\sigma_2}||\ge\delta,\end{equation}
where $\sigma_1$ and $\sigma_2$ are the children of $\tau$. A
Banach space $X$ is said to have the {\it infinite tree property}
if it contains a bounded $\delta$-tree.

A {\it $\delta$-tree of depth $n$} in a Banach space $X$ is a
finite subset $\{x_\tau\}_{\tau\in T_n}$ of $X$ labelled by the
binary tree $T_n$ of depth $n$, such that the condition
\eqref{E:TreeConditions} is satisfied for each $\tau\in T_n$,
which is not a leaf. A Banach space $X$ has the {\it finite tree
property} if for some $\delta>0$ and each $n\in\mathbb{N}$ the
unit ball of $X$ contains a $\delta$-tree of depth $n$.
\end{definition}

\begin{remark}
In the ``only if'' direction of Theorem \ref{T:JS} there is a
different (and more complicated) proof in \cite{Ost14a,Ost14+},
which consists in a combination of the following two results:

\begin{itemize}

\item[(i)] Existence of a bilipschitz embedding of the infinite
diamond $D_\omega$ into any nonseparable dual of a separable
Banach space (using Stegall's \cite{Ste75} construction), see
\cite{Ost14a}.

\item[(ii)] Finite subsets of  a metric space which admits a
bilipschitz embedding into any nonseparable dual of a separable
Banach space, admit embeddings into any nonsuperreflexive Banach
space with uniformly bounded distortions, see \cite{Ost14+}.  (The
proof uses transfinite duals \cite{DJL76,DL76,Bel82} and the
results of Brunel-Sucheston \cite{BS75,BS76} and Perrott
\cite{Per79} on equal-signs-additive sequences.)
\end{itemize}
\end{remark}

I would like to turn your attention to the fact that the
Johnson-Schechtman Theorem \ref{T:JS} shows some obstacles on the
way to a solution the (mentioned above) problem for
superreflexivity:
\smallskip

{\it Characterize metric spaces which admit bilipschitz embeddings
into some superreflexive Banach spaces.}
\medskip

We need the following definitions and results. Let
$\{M_n\}_{n=1}^\infty$ and $\{R_n\}_{n=1}^\infty$ be two sequences
of metric spaces. We say that $\{M_n\}_{n=1}^\infty$ admits {\it
uniformly bilipschitz embeddings} into $\{R_n\}_{n=1}^\infty$ if
for each $n\in\mathbb{N}$ there is $m(n)\in\mathbb{N}$ and a
bilipschitz map $f_n:M_n\to R_{m(n)}$ such that the distortions of
$\{f_n\}_{n=1}^\infty$ are uniformly bounded.

\begin{theorem}[\cite{Ost14c}]\label{T:NETreeInDiam}  Binary trees $\{T_n\}_{n=1}^\infty$ do not admit
uniformly bilipschitz embeddings into diamonds
$\{D_n\}_{n=1}^\infty$.
\end{theorem}

\begin{proof}[Without proof] The proof is elementary, but rather
lengthy combinatorial argument.\end{proof}

There is also a non-embeddability in the other direction: The fact
that diamonds $\{D_n\}$ do not admit uniformly bilipschitz
embeddings into binary trees $\{T_n\}$ is well known, it follows
immediately from the fact that $D_n$ $(n\ge 1)$ contains a cycle
of length $2^{n+1}$ isometrically, and the well-known observation
of Rabinovich and Raz \cite{RR98} stating that the distortion of
any embedding of an $m$-cycle into any tree is $\ge\frac{m}3-1$.
\medskip

\begin{remark} Mutual nonembeddability of Laakso graphs and binary
trees is much simpler: (1) Laakso graphs are non-embeddable into
trees because large Laakso graphs contain large cycles
isometrically. (2) Binary trees are not embeddable into Laakso
graphs because the Laakso graphs are uniformly doubling (see
\cite[p.~81]{Hei01} for the definition of a doubling metric
space), but binary trees are not uniformly doubling.
\end{remark}

Let us show that these results, in combination with some other
known results, imply that it is impossible to find a sequence
$\{C_n\}_{n=1}^\infty$ of finite metric spaces which admits
uniformly bilipschitz embeddings into a metric space $M$ if and
only if $M$ does not admit a bilipschitz embedding into a
superreflexive Banach space. Assume the contrary: Such sequence
$\{C_n\}_{n=1}^\infty$ exists. Then $\{C_n\}$ admits uniformly
bilipschitz embeddings into the infinite binary tree. Therefore,
by the result of Gupta \cite{Gup01}, the spaces
$\{C_n\}_{n=1}^\infty$ are uniformly bilipschitz-equivalent to
weighted trees $\{W_n\}_{n=1}^\infty$ . The trees
$\{W_n\}_{n=1}^\infty$  should admit, by a result Lee-Naor-Peres
\cite{LNP09} uniformly bilipschitz embeddings of increasing binary
trees (these authors proved that $\{W_n\}_{n=1}^\infty$  would
admit uniformly bilipschitz embeddings into $\ell_2$ otherwise).
Therefore, by Theorem \ref{T:NETreeInDiam} the spaces
$\{C_n\}_{n=1}^\infty$ cannot be embeddable into diamonds with
uniformly bounded distortion. Therefore they do not admit
uniformly bilipschitz embeddings into $D_\omega$ (since the union
of $\{D_i\}_{i=0}^\infty$ is dense in $D_\omega$). On the other
hand, Theorem \ref{T:JS} implies that $D_\omega$ does not admit a
bilipschitz embedding into a superreflexive space, a
contradiction.
\medskip

One can try to find a characterization of metric spaces which are
embeddable into superreflexive spaces in terms of some
inequalities for distances. Some hope for such characterization
was given by the already mentioned Markov convexity introduced by
Lee-Naor-Peres \cite{LNP09}, because it provides a reason for
non-embeddability into superreflexive Banach spaces of both binary
trees and diamonds (and many other trees and diamond-like spaces).

\begin{definition}[Lee-Naor-Peres \cite{LNP09}]\label{D:MarkConv}
Let $\{X_t\}_{t\in \mathbb{Z}}$ be a Markov chain on a state space
$\Omega$. Given an integer $k\ge 0$,  we denote by $\{\widetilde
X_t(k)\}_{t\in \mathbb{Z}}$ the process which equals $X_t$ for
time $t\le k$, and evolves independently (with respect to the same
transition probabilities) for time $t > k$. Fix $p>0$. A metric
space $(X,d_X)$ is called {\it Markov $p$-convex with constant
$\Pi$} if for every Markov chain $\{X_t\}_{t\in \mathbb{Z}}$ on a
state space $\Omega$, and every $f : \Omega \to X$,
\begin{equation}\label{E:MarkConv}
\sum_{k=0}^{\infty}\sum_{t\in \mathbb{Z}}\frac{\mathbb{E}\left[
d_X\left(f(X_t),f\left(\widetilde
X_t\left(t-2^{k}\right)\right)\right)^p\right]}{2^{kp}} \le \Pi^p
\cdot \sum_{t\in \mathbb{Z}}\mathbb{E}
\big[d_X(f(X_t),f(X_{t-1}))^p\big].
\end{equation}
The least constant $\Pi$ for which~\eqref{E:MarkConv} holds for
all Markov chains is called the {\it Markov $p$-convexity
constant} of $X$, and is denoted $\Pi_p(X)$. We say that $(X,d_X)$
is {\it Markov $p$-convex} if $\Pi_p(X) < \infty$.
\end{definition}

\begin{remark} The choice of the rather complicated left-hand side
in \eqref{E:MarkConv} is inspired by the original Bourgain's proof
\cite{Bou86} of the ``if'' part of Theorem \ref{T:Bourgain}.
\end{remark}

\begin{remark}\label{R:MarkovpVSq} It is unknown whether for general metric
spaces Markov $p$-convexity implies Markov $q$-convexity for
$q>p$. (This is known to be true for Banach spaces.)
\end{remark}

Lee-Naor-Peres \cite{LNP09} showed that Definition
\ref{D:MarkConv} is important for the theory of metric embeddings
by proving that each superreflexive space $X$ is Markov $q$-convex
for sufficiently large $q$. More precisely, it suffices to pick
$q$ such that $X$ has an equivalent $q$-convex norm (see
Definition \ref{D:ModConvQconv}), and by Theorem \ref{T:Pisier} of
Pisier, such $q\in[2,\infty)$ exists for each superreflexive
space.
\medskip

On the other hand, Lee-Naor-Peres have shown that for any
$0<p<\infty$ the Markov $p$-convexity constants of binary trees
$\{T_n\}$ are not uniformly bounded. Later Mendel and Naor
\cite{MN13} verified that the Markov $p$-convexity constants of
Laakso graphs are not uniformly bounded. Similar proof works for
diamonds $\{D_n\}$. See Theorem \ref{T:ThickNoMark} and Remark
\ref{R:MarkovConvDiam} for a more general result.

\begin{example}[Lee-Naor-Peres \cite{LNP09}]
For every $m \in \mathbb N$, we have $\Pi_p(T_{2^m}) \geq
2^{1-\frac{2}{p}}\cdot m^{\frac1p}.$
\end{example}

{

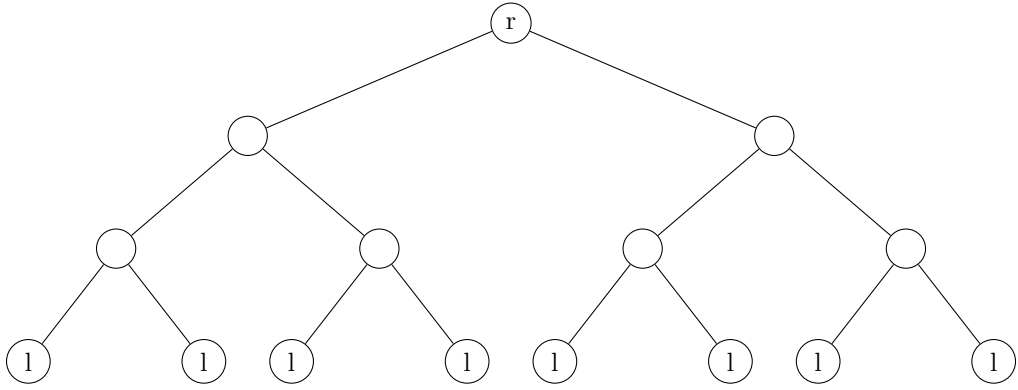
\begin{figure}
{
\begin{tikzpicture}[level/.style={sibling distance=70mm/#1}]
\node [circle,draw] (z) {\hbox{r}}
  child {node [circle,draw] (a) {\hbox{~~}}
    child {node [circle,draw] (b) {\hbox{~~}}
      child {node [circle,draw]  {\hbox{l}}}
      child {node [circle,draw]  {\hbox{l}}}
    }
    child {node [circle,draw] (g) {\hbox{~~}}
      child {node [circle,draw]  {\hbox{l}}}
      child {node [circle,draw]  {\hbox{l}}}
    }
  }
  child {node [circle,draw] (a) {\hbox{~~}}
    child {node [circle,draw] (b) {\hbox{~~}}
      child {node [circle,draw]  {\hbox{l}}}
      child {node [circle,draw]  {\hbox{l}}}
    }
    child {node [circle,draw] (g) {\hbox{~~}}
      child {node [circle,draw]  {\hbox{l}}}
      child {node [circle,draw]  {\hbox{l}}}
      }
    };
\end{tikzpicture}
} \caption{$T_3$, with the root (r) and leaves (l) marked.}
\label{F:TreeRootLeaves}
\end{figure}

}

\begin{proof} Simplifying the description of the chain somewhat (precise description of $\Omega$ and the map $f$ requires some
formalities), we consider only times $t=1,\dots, 2^m$ and let
$\{X_t\}_{t=0}^m$ be the downward random walk on $T_{2^m}$ which
is at the root at time $t=0$ and $X_{t+1}$ is obtained from $X_t$
by moving down-left or down-right with probability $\frac12$ each,
see Figure \ref{F:TreeRootLeaves}. We also assume that $X_t$ is at
the root with probability $1$ if $t<0$ (here more formal
description of the chain is needed) and that for $t>2^m$ we have
$X_{t+1}=X_t$ (this is usually expressed by saying that {\it
leaves are absorbing states}). Then
$$
\sum_{t=1}^{2^m} \mathbb E\left[d_{T_{2^m}}(X_t,X_{t-1})^p\right]
= 2^m.
$$
Moreover, in the downward random walk, after splitting at time $r
\leq 2^m$ with probability at least $\frac12$ two independent
walks will accumulate distance which is at least twice the number
of steps (until a leaf is encountered). Thus
\[\begin{split}
\sum_{k=0}^m\sum_{t=1}^{2^{m}}\frac{\mathbb{E}\left[d_{T_{2^m}}\left(X_{t},\widetilde
X_{t}\left(t-2^{k}\right)\right)^p\right]}{2^{kp}} &\geq
\sum_{k=0}^{m-1} \sum_{t=2^k}^{2^m} \frac{1}{2^{kp}} \cdot \frac12
\cdot 2^{(k+1)p}\\& \geq m2^{m-1}2^{p-1}\\&=2^{p-2}\cdot m \cdot
2^{m}.
\end{split}\] The claim follows.
\end{proof}

For diamond graphs and Laakso graphs the argument is similar, but
more complicated, because in such graphs the trajectories can come
close after separation.\medskip

After uniting the reasons for non-embeddability for diamonds and
trees one can hope to show that Markov convexity characterizes
metric spaces which are embeddable into superreflexive spaces. It
turns out that this is not the case. It was shown by Li
\cite{Li14,Li14+} that the Heisenberg group
$\mathbb{H}(\mathbb{R})$ (see Definition \ref{D:Heisenberg}) is
Markov convex. On the other hand it is known that the Heisenberg
group does not admit a bilipschitz embedding into any
superreflexive Banach space \cite{CK06,LN06}. (It is worth
mentioning that in the present context we may consider the
discrete Heisenberg group $\mathbb{H}(\mathbb{Z})$ consisting of
the matrices with integer entries of the form shown in Definition
\ref{D:Heisenberg}, endowed with its word distance, see Definition
\ref{D:GroupDef}.)\medskip

I suggest the following problem which is open as far as I know.

\begin{problem}\label{P:MarkovTestSp} Does there exist a  test-space for
superreflexivity which is Mar\-kov $p$-convex for some
$0<p<\infty$? (Or a sequence of test-spaces with uniformly bounded
Markov $p$-convexity constants?)
\end{problem}

\begin{remark} The Heisenberg group $H(\mathbb{Z})$ (with integer entries) has two properties needed for the test-space in Problem \ref{P:MarkovTestSp}: it is not embeddable into
any superreflexive space and is Markov convex. The only needed
property which it does not have is: embeddability into each
nonsuperreflexive space. Cheeger and Kleiner \cite{CK10} proved
that $H(\mathbb{Z})$ is not embeddable into some nonsuperreflexive
Banach spaces, for example, into $L_1(0,1)$.
\end{remark}

One more problem which I would like to mention here is:

\begin{problem}[Naor, July 2013]\label{P:UnivTestSp} Does there exist a
sequence of finite metric spa\-ces $\{M_i\}_{i=1}^\infty$ which is
a sequence of test-spaces for superreflexivity with the following
universality property: if $\{A_i\}_{i=1}^\infty$ is a sequence of
test-spaces for superreflexivity, then there exist uniformly
bilipschitz embeddings $E_i:A_i\to M_{n(i)}$, where
$\{n(i)\}_{i=1}^\infty$ is some sequence of positive integers?
\end{problem}

\subsection{Characterization of superreflexivity in terms of one
test-space}

Baudier \cite{Bau07} strengthened the ``only if'' direction of
Bourgain's characterization  and proved

\begin{theorem}[Baudier \cite{Bau07}]\label{T:Baudier} A Banach space $X$ is
nonsuperreflexive if and only if it admits bilipschitz embedding
of the infinite binary tree $T_\infty$.
\end{theorem}

The following result hinted that possibly the Cayley graph of any
nontrivially complicated hyperbolic group is the test-space for
superreflexivity:

\begin{theorem}[Buyalo--Dranishnikov--Schroe\-der \cite{BDS07}]\label{T:BDS07} Every Gromov hyperbolic group admits a
quasi-isometric embedding into the product of finitely many copies
of the binary tree.
\end{theorem}

Let us introduce notions used in this statement.

\begin{definition}\label{D:GroupDef} Let $G$ be a group generated by a finite set $S$.

\begin{itemize}

\item The {\it Cayley graph} $\cay(G,S)$ is defined as a graph
whose vertex set is $G$ and whose edge set is the set of all pairs
of the form $(g,sg)$, where $g\in G$, $s\in S$.

\item In this context we consider each edge as a line segment of
length $1$ and endow $\cay(G,S)$ with the shortest path distance.
The restriction of this distance to $G$ is called the {\it word
distance}.

\item  Let $u$ and $v$ be two elements in a metric space
$(M,d_M)$. A {\it $uv$-geodesic} is a distance-preserving map
$g:[0,d_M(u,v)]\to M$ such that $g(0)=u$ and $g(d_M(u,v))=v$
(where $[0,d_M(u,v)]$ is an interval of the real line with the
distance inherited from $\mathbb{R}$).

\item A metric space $M$ is {\it geodesic} if for any two points
$u$ and $v$ in $M$, there is a $uv$-geodesic in $M$; $\cay(G,S)$,
with edges identified with line segments and with the shortest
path distance is a geodesic metric space.

\item A geodesic metric space $M$ is called {\it
$\delta$-hyperbolic}, if for each triple $u,v,w\in M$ and any
choice of a $uv$-geodesic, $vw$-geodesic, and $wu$-geodesics, each
of these geodesics is in the $\delta$-neighborhood of the union of
the other two.

\item A group is {\it word hyperbolic} or {\it Gromov hyperbolic}
if $\cay(G,S)$ is $\delta$-hyperbolic for some $\delta<\infty$.
\end{itemize}
\end{definition}

\begin{remark}
\begin{itemize}
\item It might seem that the definition of hyperbolicity depends
on the choice of the generating set $S$.

\item It turns out that the value of $\delta$ depends on $S$, but
its existence does not.

\item The theory of hyperbolic groups was created by Gromov
\cite{Gro87}, although some related results were known before. The
theory of hyperbolic groups plays an important role in group
theory, geometry, and topology.

\item Theory of hyperbolic groups is presented in many sources,
see \cite{AB+91} and \cite{BH99}.

\end{itemize}
\end{remark}

\begin{remark} It is worth mentioning that the identification of edges of $\cay(G,S)$ with line segments is
useful and important when we study geodesics and introduce the
definition of hyperbolicity. In the theory of embeddings it is
much more convenient to consider $\cay(G,S)$ as a countable set
(it is countable because we consider groups generated by finite
sets), endowed with the shortest path distance (in the
graph-theoretic sense), in this context it is called the {\it word
distance}. See \cite{Ost13b,Ost14e} for relations between
embeddability of graphs as vertex sets and as geodesic metric
spaces.
\end{remark}

It is worth mentioning that although different finite generating
sets $S_1$ and $S_2$ in $G$ lead to different word distances on
$G$, the resulting metric spaces are bilipschitz equivalent: the
identity map $(G,d_{S_1})\to (G,d_{S_2})$ is bilipschitz, where
$d_{S_1}$ is the word distance corresponding to $S_1$ and
$d_{S_2}$ is the word distance corresponding to $S_2$.

We also need the following definitions used in \cite{BDS07}. A map
$f:X\to Y$ between metric spaces $(X,d_X)$ and $(Y,d_Y)$ is called
a {\it quasi-isometric embedding} if there are $a_1, a_2 > 0$ and
$b\ge 0$, such that
\begin{equation}\label{E:QuasiIsom} a_1 d_X(u,v)- b\le d_Y(f(u), f(v))\le a_2d_X(u,v) + b\end{equation} for
all $u,v\in X$. By a {\it binary tree} the authors of \cite{BDS07}
mean an infinite tree in which each vertex has degree $3$. By a
{\it product of trees}, denoted $(\oplus_{i=1}^n T(i))_1$, we mean
their Cartesian product with the $\ell_1$-metric, that is,
\begin{equation}\label{E:L_1Prod}d(\{u_i\},\{v_i\})= \sum_{i=1}^n
d_{T(i)}(u_i,v_i).\end{equation}

\begin{observation} The binary tree defined as an infinite tree in
which each vertex has degree $3$ is isometric to a subset of the
product in the sense of \eqref{E:L_1Prod} of three copies of
$T_\infty$.\end{observation}

Therefore we may replace the infinite binary tree by $T_\infty$ in
the statement of Theorem \ref{T:BDS07}. Hence the
Buyalo--Dranishnikov--Schroe\-der Theorem \ref{T:BDS07} in
combination with the Baudier theorem  \ref{T:Baudier} implies the
existence of a quasi-isometric embedding of any Gromov hyperbolic
group, which is embeddable into product of $n$ copies of
$T_\infty$, into any Banach space containing an isomorphic copy of
a direct sum of $n$ nonsuperreflexive spaces. The fact that
Buyalo-Dranishnikov-Schroeder consider quasi-isometric embeddings
(which are weaker than bilipschitz) is not a problem. One can
easily prove the following lemma. Recall that a metric space is
called {\it locally finite} if all balls of finite radius in it
have finite cardinality (a detailed proof of Lemma
\ref{L:QIBanBilip} can be found in \cite[Lemma 2.3]{Ost14c}).

\begin{lemma}\label{L:QIBanBilip} If a locally finite metric space $M$ admits a
quasi-isometric embedding into an infinite-dimensional Banach
space $X$, then $M$ admits a bilipschitz embedding into $X$.
\end{lemma}

\begin{remark} One can easily construct a counterexample to a similar statement for
general metric spaces.
\end{remark}

{\bf Back to embeddings:} However, we know from results of
Gowers-Maurey \cite{GM93} that there exist nonsuperreflexive
spaces which do not contain isomorphically direct sums of any two
infinite-dimensional Banach spaces, so we do not get immediately
bilipschitz embeddability of hyperbolic groups into
nonsuperreflexive Banach spaces.
\medskip

Possibly this obstacle can be overcome by modifying Baudier's
proof of Theorem \ref{T:Baudier} for the case of a product of
several trees, but at this point a more general result is
available. It can  be stated as:\medskip

{\it Embeddability of locally finite spaces into Banach spaces is
finitely determined.}\medskip

We need the following version of the result on finite
determination (this statement also explains what we mean by
``finite determination''):

\begin{theorem}[\cite{Ost12}]\label{T:bilip} Let $A$ be a locally finite metric space whose
finite subsets admit bilipschitz embeddings into a Banach space
$X$ with uniformly bounded distortions. Then $A$ admits a
bilipschitz embedding into $X$.
\end{theorem}

\begin{remark} This result and its version for coarse embeddings have many predecessors: Baudier \cite{Bau07,Bau12}, Baudier-Lancien
\cite{BL08}, Brown-Guentner \cite{BG05}, and Ostrovskii
\cite{Ost06,Ost09}.
\end{remark}

Now we return to embeddings of hyperbolic groups into
nonsuperreflexive spaces. Recall that we consider finitely
generated groups. It is easy to see that in this case $\cay(G,S)$
is a locally finite metric space (recall that we consider
$\cay(G,S)$ as a countable set $G$ with its word distance). By
finite determination, it suffices to show only how to embed
products of $n$ finite binary trees into an arbitrary
non-superreflexive Banach space with uniformly bounded distortions
(the distortions are allowed to grow if we increase $n$, since for
a fixed hyperbolic group the number $n$ is fixed). This can be
done using the embedding of a finite binary tree suggested by
Bourgain (Theorem \ref{T:Bourgain}) and the standard techniques
for constructions of basic sequences and finite-dimensional
decompositions. This techniques (going back to Mazur) allows to
show that for each $n$ and $N$ find a sequence of
finite-dimensional spaces $X_i$ such that $X_i$ contains a
2-bilipschitz image of $T_N$ and the direct sum $(\oplus_{i=1}^n
X_i)_1$ is $C(n)$-isomorphic to their linear span in $X$ ($C(n)$
is constant which depends on $n$, but not on $N$). See
\cite[pp.~157--158]{Ost14c} for a detailed argument.

So we have proved that each Gromov hyperbolic group admits a
bilipschitz embedding into any nonsuperreflexive Banach space.
This proves the corresponding part of the following theorem.

\begin{theorem} Let $G$ be a Gromov hyperbolic group which does
not contain a cyclic group of finite index. Then the Cayley graph
of $G$ is a test-space for superreflexivity.
\end{theorem}

The other direction follows by a combination of results of
Bourgain \cite{Bou86}, Benjamini and Schramm \cite{BS97}, and some
basic theory of hyperbolic groups \cite{BH99,NY12}, see
\cite[Remark 2.5]{Ost14c} for details.\medskip

I find the following open  problem interesting.

\begin{problem}\label{P:GroupsSR} Characterize finitely generated infinite groups whose
Cayley graphs are test-spaces for superreflexivity.
\end{problem}

Possibly Problem \ref{P:GroupsSR} is very far from its solution
and we should rather do the following. Given a group whose
structure is reasonably well understood, check\smallskip

(1) Whether it admits a bilipschitz embedding into an arbitrary
nonsuperreflexive Banach space?

(2) Whether it admits bilipschitz embeddings into some
superreflexive Banach spaces?

\begin{remark} There are groups which do not admit bilipschitz embeddings into
some nonsuperreflexive spaces, such as $L_1$. Examples which I
know:

\begin{itemize}

\item Heisenberg group (Cheeger-Kleiner \cite{CK10})

\item Gromov's random groups \cite{Gro03} containing expanders
weakly are not even coarsely embeddable into $L_1$.

\item Recently constructed groups of Osajda \cite{Osa14+} with
even stronger properties.
\end{itemize}
\end{remark}

\begin{remark} At the moment the only groups known to admit bilipschitz
embeddings into superreflexive spaces are groups containing
$\mathbb{Z}^n$ as a subgroup of finite index.
de~Cornuilier-Tessera-Valette \cite{CTV07} conjectured that such
groups are the only groups admitting a bilipschitz embedding into
$\ell_2$. This conjecture is still open. I asked about the
superreflexive version of this conjecture on MathOverflow
\cite{Ost14f} (August 19, 2014) and de~Cornuilier commented on it
as: ``In the main two cases for which the conjecture is known to
hold in the Hilbert case, the same argument also works for
arbitrary uniformly convex Banach spaces''. \end{remark}

\begin{remark}
Groups which are test-spaces for superreflexivity do not have to
be hyperbolic. In fact, one can show that a direct product of
finitely many hyperbolic groups is a test-space for
superreflexivity provided at least one of them does not have a
cyclic group as a subgroup of finite index. It is easy to check
(using the definition) that such products are not Gromov
hyperbolic unless all-except-one groups in the product are finite
(the reason is that $\mathbb{Z}^2$ is not Gromov hyperbolic).
\end{remark}

Now I would like to return to the title of this section:
``Characterization of superreflexivity in terms of one
test-space''. This can actually be done using either the Bourgain
or the Johnson-Schechtman characterization and the following
elementary proposition (I published it \cite{Ost14c}, but I am
sure that it was known to interested people):

\begin{proposition}[{\cite[Section 5]{Ost14c}}] {\bf (a)} Let $\{S_n\}_{n=1}^\infty$ be a sequence of finite test-spaces for some class $\mathcal{P}$ of Banach
spaces containing all finite-dimensional Banach spaces. Then there
is a metric space $S$ which is a test-space for $\mathcal{P}$.
\medskip

{\bf (b)} If $\{S_n\}_{n=1}^\infty$ are

\begin{itemize}

\item unweighted graphs,

\item trees,

\item graphs with uniformly bounded degrees,

\end{itemize}
then $S$ also can be required to have the same property.
\end{proposition}

\begin{proof}[Sketch of the proof] In all of the cases the constructed space $S$ contains subspaces
isometric to each of $\{S_n\}_{n=1}^\infty$. Therefore the only
implication which is nontrivial is that the embeddability of
$\{S_n\}_{n=1}^\infty$ implies the embeddability of $S$.
\medskip

Each finite metric space can be considered as a weighted graph
with its shortest path distance. We construct the space $S$ as an
infinite graph by joining $S_n$ with $S_{n+1}$ with a path $P_n$
whose length is $\ge\max\{\diam S_n, \diam S_{n+1}\}$. To be more
specific, we pick in each $S_n$ a vertex $O_n$ and let $P_n$ be a
path joining $O_n$ with $O_{n+1}$. We endow the infinite graph $S$
with its shortest path distance. It is clear that
$\{S_n\}_{n=1}^\infty$ embed isometrically into $S$ and all of the
conditions in (b) are satisfied. It remains only to show that each
infinite-dimensional Banach space $X$ which admits bilipschitz
embeddings of $\{S_n\}_{n=1}^\infty$ with uniformly bounded
distortions, admits a bilipschitz embedding of $S$. This is done
by embedding $S_n$ into any hyperplane of $X$ with uniformly
bounded distortions. This is possible because the sets are finite,
the space is infinite-dimensional, and all hyperplanes in a Banach
space are isomorphic with the Banach-Mazur distances being $\le $
some universal constant.
\medskip

Now we consider in $X$ parallel hyperplanes $\{H_n\}$ with the
distance between $H_n$ and $H_{n+1}$ equal to the length of $P_n$
and embed everything in the corresponding way. All computations
are straightforward (see \cite{Ost14c} for details).\end{proof}

\section{Non-local properties}

One can try to find metric characterizations of classes of Banach
spaces which are not {\it local} in the sense that the conditions
(1) $X\in\mathcal{P}$ and (2) $Y$ is finitely representable in
$X$, do not necessarily imply that $Y\in\mathcal{P}$. Apparently
this study should not be considered as a part of the Ribe program,
and this direction has developed much more slowly than the
directions related to the Ribe program. It is clear that even if
we restrict our attention to properties which are hereditary
(inherited by closed subspaces) and isomorphic invariant, the
class of non-local properties which have been already studied in
the literature is huge. I found in the literature only four
properties for which the problem of metric characterization was
ever considered.\medskip

\subsection{Asymptotic uniform convexity and smoothness}

One of the first results of the described type is the following
result of Baudier-Kalton-Lancien, where by $T_\infty^\infty$ we
denote the tree defined similarly to the tree $T_\infty$, but now
we consider all possible finite sequences with terms in
$\mathbb{N}$, and so degrees of all vertices of $T_\infty^\infty$
are infinite.

\begin{theorem}[\cite{BKL10}]\label{T:BKL10} Let $X$ be a reflexive Banach space. The following assertions
are equivalent:
\begin{itemize}

\item ${T}_\infty^\infty$ admits a bilipschitz embedding into $X$.

\item $X$ does not admit any equivalent asymptotically uniformly
smooth norm \underline{or} $X$ does not admit any equivalent
asymptotically uniformly convex norm.

\item The Szlenk index of $X$ is $>\omega$ or the Szlenk index of
$X^*$ is $>\omega$, where $\omega$ is the first infinite ordinal.
\end{itemize}
\end{theorem}

It is worth mentioning that Dilworth, Kutzarova, Lancien, and
Randrianarivony \cite{DKLR14} found an interesting geometric
description of the class of Banach spaces whose metric
characterization is provided by Theorem \ref{T:BKL10}.

\subsection{Radon-Nikod\'ym property}\label{S:RNP}

The {\it Radon-Nikod\'ym property} (\RNP) is one of the most
important isomorphic invariants of Banach spaces. This class also
plays an important role in the theory of metric embeddings, this
role is partially explained by the fact that for this class one
can use differentiability to prove non-embeddability
results.\medskip

There are many expository works presenting results on the \RNP, we
recommend the readers (depending on the taste and purpose) one of
the following sources \cite[Chapter 5]{BL00}, \cite{Bou79},
\cite{Bou83}, \cite{DU77}, \cite{Dul85}, \cite{Pis14+}.

\subsubsection{Equivalent definitions of \RNP}

One of the reasons for the importance of the \RNP\ is the
possibility to characterize (define) the \RNP\ in many different
ways. I would like to remind some of them:

\begin{itemize}

\item  Measure-theoretic definition (it gives the name to this
property) $X\in\RNP\Leftrightarrow$ The following analogue of the
Radon-Nikod\'ym theorem holds for $X$-valued measures.

\begin{itemize}

\item Let $(\Omega,\Sigma,\mu)$ be a positive finite real-valued
measure, and $(\Omega,\Sigma,\tau)$ be an $X$-valued measure on
the same $\sigma$-algebra which is absolutely continuous with
respect to $\mu$ (this means $\mu(A)=0\Rightarrow \tau(A)=0$) and
satisfies the condition $\tau(A)/\mu(A)$ is a uniformly bounded
set of vectors over all $A\in\Sigma$ with $\mu(A)\ne0$. Then there
is an $f\in L_1(\mu,X)$ such that
$$\forall A\in\Sigma\quad\tau(A)=\int_A f(\omega)d\mu(\omega).$$

\end{itemize}

\item Definition in terms of differentiability (goes back to
Clarkson \cite{Cla36} and Gel\-fand \cite{Gel38})
$X\in\RNP\Leftrightarrow$ $X$-valued Lipschitz functions on
$\mathbb{R}$ are differentiable almost everywhere.

\item Probabilistic definition (Chatterji \cite{Cha68})
$X\in\RNP\Leftrightarrow$ Bounded $X$-valued martingales converge.

\begin{itemize}

\item In more detail:  A Banach space $X$ has the RNP if and only
if each $X$-valued martingale $\{f_n\}$ on some probability space
$(\Omega,\Sigma,\mu)$, for which $\{||f_n(\omega)||:~
n\in\mathbb{N},~ \omega\in\Omega\}$ is a bounded set, converges in
$L_1(\Omega,\Sigma, \mu, X)$.

\end{itemize}

\item Geometric definition. $X\in\RNP\Leftrightarrow$ Each bounded
closed convex set in $X$ is dentable in the following sense:

\begin{itemize}

\item  A bounded closed convex subset $C$ in a Banach space $X$ is
called {\it dentable} if for each $\ep>0$ there is a continuous
linear functional $f$ on $X$ and $\alpha>0$ such that the set
\[\left\{y\in C:~f(y)\ge\sup\{f(x):~ x\in C\}-\alpha\right\}\]
has diameter $<\ep$.

\end{itemize}

\item {\bf Examples {\rm (these lists are far from being
exhaustive)}:}

\begin{itemize}

\item \RNP: Reflexive (for example $L_p$, $1<p<\infty$), separable
dual spaces (for example, $\ell_1$).

\item non-\RNP: $c_0$, $L_1(0,1)$, nonseparable duals of separable
Banach spaces.

\end{itemize}

\end{itemize}

\subsubsection{RNP and metric embeddings}

Cheeger-Kleiner \cite{CK06} and Lee-Naor \cite{LN06} noticed that
the observation of Sem\-mes \cite{Sem96} on the result of Pansu
\cite{Pan89} can be generalized to maps of the Heisenberg group
into Banach spaces with the RNP. This implies that Heisenberg
group with its subriemannian metric (see Definition
\ref{D:Heisenberg}) does not admit a bilipschitz embedding into
any space with the RNP.

Cheeger-Kleiner \cite{CK09} generalized some part of
differentiability theory of Chee\-ger \cite{Che99} (see also
\cite{Kei04,KM11}) to maps of metric spaces into Banach spaces
with the RNP. This theory implies some non-embeddability results,
for example it implies that the Laakso space does not admit a
bilipschitz embedding into a Banach space with the \RNP.
\medskip

\subsubsection{Metric characterization of \RNP}

In 2009 Johnson \cite[Problem 1.1]{Tex09} suggested the problem:
Find a purely metric characterization of the Radon-Nikod\'ym
property (that is, find a characterization of the \RNP\ which does
not refer to the linear structure of the space). The main goal of
the rest of Section \ref{S:RNP} is to present such
characterization.
\medskip

It turns out that the \RNP\ can be characterized in terms of {\it
thick families of geodesics} defined as follows (different
versions of this definition appeared in
\cite{Ost14a,Ost14b,Ost14d}, the following seems to be the most
suitable definition).

\begin{definition}[\cite{Ost14a,Ost14b}]\label{D:ThickFam} A family $T$ of $uv$-geodesics is called
{\it thick} if there is $\alpha>0$ such that for every $g\in T$
and for every finite collection of points $r_1,\dots,r_n$ in the
image of $g$, there is another $uv$-geodesic $\widetilde g\in T$
satisfying the conditions:

\begin{itemize}

\item[(i)] The image of $\widetilde g$ also contains
$r_1,\dots,r_n$ (we call these points {\it control points}).

\item[(ii)] Possibly there are some more common points of $g$ and
$\widetilde g$.

\item[(iii)]  There is a sequence $0=q_0< s_1< q_1< s_2<
q_2<\dots< s_m<q_m=d_M(u,v)$, such that $g(q_i)=\widetilde g(q_i)$
($i=0,\dots, m$) are common points containing $r_1,\dots,r_n$; and
$\sum_{i=1}^{m}d_M(g(s_i),\widetilde g(s_i))\ge\alpha.$

\item[(iv)] Furthermore, each geodesic which on some intervals
between the points $0=q_0<q_1<q_2<\dots<q_m=d_M(u,v)$ coincides
with $g$ and on others with $\widetilde g$ is also in $T$.

\end{itemize}

\end{definition}

\begin{example} Interesting and important examples of spaces
having thick families of geodesics are the infinite diamond
$D_\omega$ and the Laakso space $L_\omega$, but now we consider
them not as unions of finite sets, but as unions of geodesic
metric spaces obtained from weighted  $\{D_n\}_{n=0}^\infty$ and
$\{L_n\}_{n=0}^\infty$ in which edges are identified with line
segments of lengths $\{2^{-n}\}_{n=0}^\infty$ and
$\{4^{-n}\}_{n=0}^\infty$, respectively. Observe that for such
graphs there are also natural (although non-unique) isometric
embeddings of $D_n$ into $D_{n+1}$ and $L_n$ into $L_{n+1}$, and
therefore the unions are well-defined. It is easy to check that
the families of all geodesics in $D_\omega$ and $L_\omega$ joining
the vertices of $D_0$ and $L_0$, respectively, are thick.
\end{example}

\begin{theorem}[\cite{Ost14a}]\label{T:ThickFam} A Banach space $X$ does not have the
RNP if and only if  there exists a metric space $M_X$ containing a
thick family $T_X$ of geodesics which admits a bilipschitz
embedding into $X$.
\end{theorem}

\begin{remark} Theorem \ref{T:ThickFam} implies the result of
Cheeger and Kleiner \cite{CK09} on nonexistence of bilipschitz
embeddings of the Laakso space into Banach spaces with the \RNP.
\end{remark}

It turns out that the metric space $M_X$ whose existence is
established in Theorem \ref{T:ThickFam} cannot be chosen
independently of $X$, because the following result holds.

\begin{theorem}[\cite{Ost14a}]\label{T:Customize} For each metric space
$M$ containing a thick family of geodesics there exists a Banach
space $X$ which does not have the \RNP\ and does not admit a
bilipschitz embedding of $M$.
\end{theorem}

Because of Theorem \ref{T:Customize} the following is an open
problem:

\begin{problem}\label{P:TestRNP} Can we characterize the \RNP\ using test-spaces?
\end{problem}

Also I would like to mention the problem of the metric
characterization of the \RNP\ can have many different (correct)
answers, so it is natural to try to find metric characterizations
of the \RNP\ in some other terms.

\begin{itemize}

\item Proof of Theorem \ref{T:ThickFam} (in both directions) is
based on the characterization of the \RNP\ in terms of
martingales. It will be presented in Section \ref{S:RNPProof}.

\item It is not true that each Banach space without \RNP\ contains
a thick family of geodesics, because Banach spaces without \RNP\
can have the uniqueness of geodesics property (consider a strictly
convex renorming of a separable Banach space without \RNP), so the
words `bilipschitz embedding' in Theorem \ref{T:ThickFam} cannot
be replaced by `isometric embedding'.

\item Proof of Theorem \ref{T:Customize} is based on the
construction of Bourgain and Rosenthal \cite{BR80} of `small'
subspaces of $L_1(0,1)$ which still do not have the
Radon-Nikod\'ym property.

\end{itemize}

\begin{itemize}

\item  Studying metric characterizations of the RNP, it would be
much more useful and interesting to get a characterization of all
metric spaces which do not admit bilipschitz embeddings into
Banach spaces with the RNP.

\item In view of Theorem \ref{T:ThickFam} it is natural to ask:
whether the presence of bilipschitz images of thick families of
geodesics characterizes metric spaces which do not admit
bilipschitz embeddings into Banach spaces with the RNP?

\item It is clear that the answer to this question in full
generality is negative: we may just consider a dense subset of a
Banach space without the RNP which does not contain any continuous
curves.

\item So we restrict our attention to spaces containing
sufficiently large collections of continuous curves. Our next
result is a negative answer even in the case of geodesic metric
spaces. Recall a metric space is called {\it geodesic} if any two
points in it are joined by a geodesic.

\end{itemize}

\begin{theorem}[\cite{Ost14d}]\label{T:ThFamMSNo} There exist geodesic metric
spaces which satisfy the following two conditions simultaneously:

\begin{itemize}

\item  Do not contain bilipschitz images of thick families of
geodesics.

\item  Do not admit bilipschitz embeddings into Banach spaces with
the Ra\-don-Niko\-d\'ym property.
\end{itemize}

\end{theorem}

In \cite{Ost14d} it was shown that the Heisenberg group with its
subriemannian metric is an example of such metric space. Let us
recall the corresponding definitions.

\begin{definition}\label{D:Heisenberg}
The {\it Heisenberg group} $\mathbb{H}(\mathbb{R})$ can be defined
as the group of real upper-triangular matrices with $1$'s on the
diagonal:

\[\left[\begin{array}{ccc} 1 & x & z \\
0 & 1 & y \\
0 & 0 & 1
\end{array}\right].\]

One of the ways to introduce the {\it subriemannian metric on
$\mathbb{H}(\mathbb{R})$} is to find the tangent vectors of the
curves produced by left translations in $x$ and in $y$ directions,
that is,

\[\left.\frac{d}{d\ep}\right|_{\ep=0}\left[\begin{array}{ccc} 1 & \ep & 0 \\
0 & 1 & 0 \\
0 & 0 & 1
\end{array}\right]
\left[\begin{array}{ccc} 1 & x & z \\
0 & 1 & y \\
0 & 0 & 1
\end{array}\right]=\left[\begin{array}{ccc} 0 & 1 & y \\
0 & 0 & 0 \\
0 & 0 & 0
\end{array}\right]\]

\[\left.\frac{d}{d\ep}\right|_{\ep=0}\left[\begin{array}{ccc} 1 & 0 & 0 \\
0 & 1 & \ep \\
0 & 0 & 1
\end{array}\right]
\left[\begin{array}{ccc} 1 & x & z \\
0 & 1 & y \\
0 & 0 & 1
\end{array}\right]=\left[\begin{array}{ccc} 0 & 0 & 0 \\
0 & 0 & 1 \\
0 & 0 & 0
\end{array}\right]\]

We introduce the distance between $u,v\in \mathbb{H}(\mathbb{R})$
as the infimum of lengths of differentiable curves joining $u$ and
$v$ with the restriction that the tangent vector at each point of
the curve is a linear combination of the two tangent vectors
computed above.
\end{definition}

This metric has been systematically studied (see
\cite{CDPT07,Gro96,Mon02}), it has very interesting geometric
properties. The Heisenberg group $\mathbb{H}(\mathbb{R})$ with its
subriemannian metric is a very important example for Metric
Geometry and its applications to Computer Science. One of the
reasons for this is its poor embeddability into many classes of
Banach spaces. As we already mentioned, Cheeger-Kleiner
\cite{CK06} and Lee-Naor \cite{LN06} proved that the Heisenberg
group does not admit a bilipschitz embedding into a Banach space
with the \RNP. It remains to show that it does not admit a
bilipschitz embedding of a thick family of geodesics.

\begin{remark} It is not needed for our argument, but is worth mentioning that

\begin{itemize}

\item Cheeger-Kleiner \cite{CK10} proved that
$\mathbb{H}(\mathbb{R})$ does not admit a bilipschitz embedding
into $L_1(0,1)$.

\item Cheeger-Kleiner-Naor \cite{CKN11} found quantitative
versions of the previous result for embeddings of finite subsets
of $\mathbb{H}(\mathbb{R})$  into $L_1(0,1)$. These quantitative
results are important for Theoretical Computer Science.
\end{itemize}
\end{remark}

We finish the proof of Theorem \ref{T:ThFamMSNo} by using the
notion of Markov convexity (Definition \ref{D:MarkConv}), proving

\begin{theorem}[\cite{Ost14d}]\label{T:ThickNoMark} A metric space with a
thick family of geodesics is not Markov $p$-convex for any
$p\in(0,\infty)$.
\end{theorem}

\noindent and combining it with the following result

\begin{theorem}[\cite{Li14,Li14+}] The Heisenberg group
$\mathbb{H}(\mathbb{R})$ is Markov $4$-convex.
\end{theorem}

\begin{remark}\label{R:MarkovConvDiam} Since the infinite diamond $D_\omega$ and the
Laakso space $L_\omega$ contain thick families of geodesics, they
are not Markov $p$-convex for any $p\in(0,\infty)$. Since the
unions of $\{D_n\}_{n=0}^\infty$ and $\{L_n\}_{n=0}^\infty$
(considered as finite sets) are dense in $D_\omega$ and
$L_\omega$, respectively; we conclude that  Markov $p$-convexity
constants of diamond graphs and Laakso graphs are not uniformly
bounded for any $p\in(0,\infty)$.
\end{remark}

\begin{remark} It is worth mentioning that the discrete
Heisenberg group $\mathbb{H}(\mathbb{Z})$ embeds into a Banach
space with the \RNP. Since $\mathbb{H}(\mathbb{Z})$ is locally
finite, this follows by combining the well-known observation of
Fr\'echet on isometric embeddability of any $n$-element set into
$\ell_\infty^n$ (see \cite[p.~6]{Ost13a}) with the finite
determination (Theorem \ref{T:bilip}, actually the earlier result
of \cite{BL08} suffices here). In fact, these results imply
bilipschitz embeddability of $\mathbb{H}(\mathbb{Z})$ into the
direct sum $(\oplus_{n=1}^\infty \ell_\infty^n)_2$, which has the
\RNP\ because it is reflexive.
\end{remark}

\subsubsection{Proof of Theorem
\ref{T:ThickFam}}\label{S:RNPProof}

First we prove: {\bf No \RNP\ $\Rightarrow$ bilipschitz
embeddability of a thick family of geodesics.}
\medskip

We need to define a more general structure than that of a
$\delta$-tree (see Definition \ref{D:DelTrees}), in which each
element is not a midpoint of a line segment, but a convex
combination.

\begin{definition}\label{D:DelBush}
Let $Z$ be a Banach space and let $\delta>0$. A set of vectors
$\{z_{n,j}\}_{n=0,j=1}^{~\infty~~m_n}$ in $Z$ is called a
$\delta$-{\it bush} if $m_0=1$ and for every $n\ge 1$ there is a
partition $\{A^n_k\}_{k=1}^{m_{n-1}}$ of $\{1,\dots,m_n\}$ such
that
\begin{equation}
||z_{n,j}-z_{n-1,k}||\ge \delta
\end{equation}
for every  $n\ge 1$ and for every $j\in A^n_k$, and
\begin{equation}\label{E:Bush2} z_{n-1,k}=\sum_{j\in
A^n_k}\lambda_{n,j}z_{n,j}\end{equation} for some
$\lambda_{n,j}\ge 0$, $\sum_{j\in A^n_k}\lambda_{n,j}=1$.
\end{definition}

\begin{theorem}\label{T:BushRNP}  A Banach space $Z$ does not have the \RNP\ if and only
if it contains a bounded $\delta$-bush for some $\delta>0$.
\end{theorem}

\begin{remark} Theorem \ref{T:BushRNP} can be derived from Chatterji's result
\cite{Cha68}. Apparently Theorem \ref{T:BushRNP} was first proved
by James, possibly even before Chatterji, see \cite{Jam81}.
\end{remark}

\begin{itemize}
\item We construct a suitable thick family of geodesics using a
bounded $\delta$-bush in a Banach space without the \RNP.

\item It is not difficult to see (for example, using the
Clarkson-Gelfand characterization) that a subspace of codimension
$1$ (hyperplane) in a Banach space without the \RNP\ does not have
the \RNP.

\item Let $X$ be a non-\RNP\ Banach space. We pick a norm-one
vector $x\in X$, then a norm-one functional $x^*$ on $X$
satisfying $x^*(x)=1$. Then (by the previous remark) we find a
bounded $\delta$-bush (for some $\delta>0$) in the kernel $\ker
x^*$. We shift this bush adding $x$ to each of its elements, and
get a (still bounded) $\delta$-bush $\{x_{i,j}\}$ satisfying the
condition $x^*(x_{i,j})=1$ for each $i$ and $j$.

\item Now we change the norm of $X$ to equivalent. The purpose of
this step is to get a norm for which we are able to construct the
thick family of geodesics in $X$, and there will be no need in a
bilipschitz embedding. (One can easily see that this would be
sufficient to prove the theorem.)

\item The unit ball of the norm $||\cdot||_1$ is defined as the
closed convex hull of the unit ball in the original norm and the
set of vectors $\{\pm x_{i,j}\}$ (recall that $\{x_{i,j}\}$ form a
bush in the hyperplane $\{x:~x^*(x)=1\}$). It is easy to check
that in this new norm the set $\{x_{i,j}\}$ is a bounded
$\delta$-bush (possibly with a somewhat smaller $\delta>0$, but we
keep the same notation). Also in the new norm we have
$||x_{i,j}||_1=1$ for all $i$ and $j$. For simplicity of notation
we shall use $||\cdot||$ to denote the new norm.
\end{itemize}

\begin{itemize}
\item We are going to use this $\delta$-bush to construct a thick
family $T_X$ of geodesics in $X$ joining $0$ and $x_{0,1}$. First
we construct a subset of the desired set of geodesics, this subset
will be constructed as the set of limits of certain broken lines
in $X$ joining $0$ and $x_{0,1}$. The constructed broken lines are
also geodesics (but they do not necessarily belong to the family
$T_X$).

\item The mentioned above broken lines will be constructed using
representations of the form $x_{0,1}=\sum_{i=1}^mz_i$, where $z_i$
are such that $||x_{0,1}||=\sum_{i=1}^m||z_i||$. The broken line
represented by such finite sequence $z_1,\dots,z_m$ is obtained by
letting $z_0=0$ and joining $\sum_{i=0}^kz_i$  with
$\sum_{i=0}^{k+1}z_i$ with a line segment for $k=0,1,\dots,m-1$.
Vectors $\sum_{i=0}^k z_i$, $k=0,1,\dots,m$ will be called {\it
vertices} of the broken line.

\item The infinite set of broken lines which we construct is
labelled by vertices of the infinite binary tree $T_\infty$ in
which each vertex is represented by a finite (possibly empty)
sequence of $0$ and $1$.
\end{itemize}

\begin{itemize}
\item The broken line corresponding to the empty sequence
$\emptyset$ is represented by the one-element sequence $x_{0,1}$,
so it is just a line segment joining $0$ and $x_{0,1}$.

\item We have
\[x_{0,1}=\lambda_{1,1}x_{1,1}+\dots+\lambda_{1,m_1}x_{1,m_1},\]
where $||x_{1,j}-x_{0,1}||\ge \delta$. We introduce the vectors
\[y_{1,j}=\frac12(x_{1,j}+x_{0,1}).\]

\item For these vectors we have
\[x_{0,1}=\lambda_{1,1}y_{1,1}+\dots+\lambda_{1,m_1}y_{1,m_1},\]
$||y_{1,j}-x_{1,j}||=||y_{1,j}-x_{0,1}||\ge\frac{\delta}2$, and
$||y_{1,j}||=1$.

\end{itemize}

\begin{itemize}
\item As a preliminary step to the construction of the broken
lines corresponding to one-element sequences $(0)$ and $(1)$ we
form a broken line represented by the points
\begin{equation}\label{E:1stSeq}\lambda_{1,1}y_{1,1},\dots,\lambda_{1,m_1}y_{1,m_1}.\end{equation}
We label the broken line represented by \eqref{E:1stSeq} by
$\overline{\emptyset}$.

\item The broken line corresponding to the one-element sequence
$(0)$ is represented by the sequence obtained from
\eqref{E:1stSeq} if we replace each term $\lambda_{1,j}y_{1,j}$ by
a two-element sequence
\begin{equation}\label{E:(0)}\frac{\lambda_{1,j}}2x_{0,1},\frac{\lambda_{1,j}}2x_{1,j}.\end{equation}

\item The broken line corresponding to the one-element sequence
$(1)$ is represented by the sequence obtained from
\eqref{E:1stSeq} if we replace each term $\lambda_{1,j}y_{1,j}$ by
a two-element sequence
\begin{equation}\label{E:(1)}\frac{\lambda_{1,j}}2x_{1,j}, \frac{\lambda_{1,j}}2x_{0,1}.\end{equation}

\end{itemize}

\begin{itemize}
\item At this point one can see where are we going to get the
thickness property from.

\item In fact, one of the inequalities above is
$||x_{1,j}-x_{0,1}||\ge \delta$. Therefore
\[\left\|\frac{\lambda_{1,j}}2x_{1,j}-\frac{\lambda_{1,j}}2x_{0,1}\right\|\ge
\frac{\lambda_{1,j}}2\,\delta.\] Summing over all $j$, we get that
the total sum of deviations  is $\ge\frac{\delta}2$.

\item In the obtained broken lines each line segment corresponds
either to a multiple of $x_{0,1}$ or to a multiple of some
$x_{1,j}$. In the next step we replace each such line segment by a
broken line. Now we describe how we do this.

\end{itemize}

\begin{itemize}
\item Broken lines corresponding to $2$-element sequences are also
formed in two steps. To get the broken lines labelled by $(0,0)$
and $(0,1)$ we apply the described procedure to the geodesic
labelled $(0)$, to get the broken lines labelled by $(1,0)$ and
$(1,1)$ we apply the described procedure to the geodesic labelled
$(1)$.\medskip

\item In the preliminary step we replace each term of the form
$\frac{\lambda_{1,k}}2x_{0,1}$ by a multiplied by
$\frac{\lambda_{1,k}}2$ sequence
$\lambda_{1,1}y_{1,1},\dots,\lambda_{1,m_1}y_{1,m_1}$, and we
replace a term of the form  $\frac{\lambda_{1,k}}2x_{1,k}$ by the
multiplied by $\frac{\lambda_{1,k}}2$ sequence
\begin{equation}
\{\lambda_{2,j}y_{2,j}\}_{j\in A^2_k},
\end{equation}
ordered arbitrarily, where $y_{2,j}=\frac{x_{1,k}+x_{2,j}}2$ and
$\lambda_{2,j}$, $x_{2,j}$, and $A^2_k$ are as in the definition
of the $\delta$-bush (it is easy to check that in the new norm we
have $||y_{2,j}||=1$). We label the obtained broken lines by
$\overline{(0)}$ and $\overline{(1)}$, respectively.

\end{itemize}

\begin{itemize}
\item To get the sequence representing the broken line labelled by
$(0,0)$ we do the following operation with the preliminary
sequence labelled $\overline{(0)}$.

\begin{itemize}

\item Replace each multiple $\lambda y_{1,j}$ present in the
sequence by the two-element sequence
\begin{equation}\label{E:Casey1}\lambda\frac{x_{0,1}}2,\lambda\frac{x_{1,j}}2.\end{equation}

\item Replace each multiple $\lambda y_{2,j}$, with $j\in A^2_k$,
present in the sequence by the two-element sequence
\begin{equation}\label{E:Casey2}\lambda\frac{x_{1,k}}2,\lambda\frac{x_{2,j}}2.\end{equation}
\end{itemize}

\item To get the sequence representing the broken line labelled by
$(0,1)$ we do the same but changing the order of terms in
\eqref{E:Casey1} and \eqref{E:Casey2}. To get the sequences
representing the broken lines labelled by $(1,0)$ and $(1,1)$, we
apply the same procedure to the broken line labelled
$\overline{(1)}$.
\end{itemize}

\begin{itemize}
\item We continue in an ``obvious'' way and get broken lines for
all vertices of the infinite binary tree $T_\infty$. It is not
difficult to see that vertices of a broken line corresponding to
some vertex $(\theta_1,\dots,\theta_n)$ are contained in the
broken line corresponding to any extension
$(\theta_1,\dots,\theta_m)$ of $(\theta_1,\dots,\theta_n)$ $(m>n)$

\item This implies that broken lines corresponding to any ray
(that is, a path infinite in one direction) in $T_\infty$ has a
limit (which is not necessarily a broken line, but is a geodesic),
and limits corresponding to two different infinite paths have
common points according to the number of common
$(\theta_1,\dots,\theta_n)$ in the vertices of those paths.

\item A thick family of geodesics is obtained by pasting pieces of
these geodesics in all ``reasonable'' ways. All verifications are
straightforward. See the details in \cite{Ost14b}.

\end{itemize}

It remains to prove:
\smallskip

\centerline{\bf Bilipschitz embeddability of a thick family of
geodesics $\Rightarrow$ No \RNP.}

\begin{proof} We assume that a metric space $(M,d)$ with a thick family of geodesics admits a bilipschitz embedding $f:M\to X$
into a Banach space $X$ and show that there exists a bounded
divergent martingale  $\{M_i\}_{i=0}^\infty$ on $(0,1]$ with
values in $X$. We assume that
\begin{equation}\label{E:Bilip}
\ell d(x,y)\le ||f(x)-f(y)||_X\le d(x,y)
\end{equation}
for some $\ell>0$. We assume that the thick family consists of
$uv$-geodesics for some $u,v\in M$ and that $d(u,v)=1$ (dividing
all distances in $M$ by $d(u,v)$, if necessary).
\medskip

Each function in the martingale $\{M_i\}_{i=0}^\infty$ will be
obtained in the following way. We consider some finite sequence
$V=\{v_i\}_{i=0}^m$ of points on any $uv$-geodesic, satisfying
$v_0=u$, $v_m=v$ and $d(u,v_{k+1})\ge d(u,v_k)$. We define $M_V$
as the function on $(0,1]$ whose value on the interval
$(d(u,v_k),d(u,v_{k+1})]$ is equal to
\[\frac{f(v_{k+1})-f(v_k)}{d(v_k,v_{k+1})}.\]

It is clear that the bilipschitz condition \eqref{E:Bilip} implies
that $||M_V(t)||\le 1$ for any collection $V$ and any $t\in(0,1]$.
Since $\{v_i\}$ are on a geodesic, is clear that an infinite
collection of such functions $\{M_{V(k)}\}_{k=0}^\infty$ forms a
martingale if for each $k\in\mathbb{N}$ the sequence $V(k)$
contains $V(k-1)$ as a subsequence. So it remains to to find a
collection of sequences $\{V(k)\}_{k=0}^\infty$ for which the
martingale $\{M_{V(k)}\}_{k=0}^\infty$  diverges. We denote
$M_{V(k)}$ by $M_k$.\medskip

Now we describe some of the ideas of the construction.

\begin{itemize}

\item It suffices to have differences $||M_k-M_{k-1}||$ to be
bounded away from zero for some infinite set of values of $k$.

\item On steps for which we achieve such estimates from below we
add exactly one new point $z_j'$ into $V(k)$ between any two
consequent points $w_{j-1}$ and $w_j$ of $V(k-1)$. In such a case
it suffices to make the choice of points in such a way that the
values of $M_{k}$ on the intervals corresponding to pairs
$(w_{j-1},z_j')$ and  $(z_j',w_j)$ are `far' from each other, and
thus from the value of $M_{k-1}$ corresponding to $(w_{j-1},w_j)$.
Actually we do not need this condition for each pair
$(w_{j-1},w_j)$, but only ``on average''.

\item Using the definition of a thick family of geodesics and the
bilipschitz condition, we can achieve this goal. A detailed
description follows.

\end{itemize}

We let $V(0)=\{u,v\}$ and so $M_0$ is a constant function on
$(0,1]$ taking value $f(v)-f(u)$. In the next step we apply the
condition of the definition of a thick family to control points
$\{u,v\}$ and any geodesic $g$ of the family. We get another
geodesic $\widetilde g$, the corresponding sequence of common
points  $\{w_i\}_{i=0}^m$ and the corresponding pair of
sufficiently well separated sequences $\{z_i,\widetilde
z_i\}_{i=1}^{m}$ on the geodesics $g$ and $\widetilde g$. The
separation condition is $\sum_{i=1}^{m}d(z_i,\widetilde
z_i)\ge\alpha$.

We let $V(1)=\{w_i\}_{i=0}^m$. Observe, that in this step we
cannot claim any nontrivial estimates for $||M_1-M_0||_{L_1(X)}$
from below because we have not made any nontrivial assumptions on
this step of the construction (it can even happen that $M_1=M_0$).
Lower estimates for martingale differences in our argument are
obtained only for differences of the form
$||M_{2k}-M_{2k-1}||_{L_1(X)}$.
\medskip

We choose $V(2)$ to be of the form
\begin{equation}\label{E:MartSeq} w_0, z'_1, w_1, z'_2, w_n,\dots,
z'_m,w_m,\end{equation} where each $z'_i$ is either $z_i$ or
$\widetilde z_i$ depending on the behavior of the mapping $f$. We
describe this dependence below. Observe that since $z_i$ or
$\widetilde z_i$ are images of the same point in $[0,1]$, the
corresponding partition of the interval $(0,1]$ does not depend on
our choice.\medskip

To make the choice of $z_i'$ we consider the quadruple $w_{i-1},
z_i, w_i, \widetilde z_i$. The bilipschitz condition
\eqref{E:Bilip} implies $||f(z_i)-f(\widetilde z_i)||\ge\ell
d(z_i,\widetilde z_i)$. Consider two pairs of vectors
corresponding to two different choices of $z'_i$:
\medskip

\noindent{\bf Pair 1:} $f(w_i)-f(z_i)$, $f(z_i)-f(w_{i-1})$.\qquad
{\bf Pair 2:} $f(w_i)-f(\widetilde z_i)$, $f(\widetilde
z_i)-f(w_{i-1})$.
\medskip

The inequality $||f(z_i)-f(\widetilde z_i)||\ge\ell
d(z_i,\widetilde z_i)$ implies that at least one of the following
is true
\begin{equation}\label{E:z}\begin{split}\left\|\frac{f(w_i)-f(z_i)}{d(w_i,z_i)}-\frac{f(z_i)-f(w_{i-1})}{d(z_i,w_{i-1})}\right\|
\ge\frac{\ell}2\, d(z_i,\widetilde
z_i)\left(\frac1{d(w_i,z_i)}+\frac1{d(z_i,w_{i-1})}\right)\end{split}
\end{equation}
or
\begin{equation}\label{E:zTilde}\begin{split}\left\|\frac{f(w_i)-f(\widetilde
z_i)}{d(w_i,\widetilde z_i)}-\frac{f(\widetilde
z_i)-f(w_{i-1})}{d(\widetilde
z_i,w_{i-1})}\right\|\ge\frac{\ell}2\, d(z_i,\widetilde
z_i)\left(\frac1{d(w_i,\widetilde z_i)}+\frac1{d(\widetilde
z_i,w_{i-1})}\right).\end{split}
\end{equation}

Since in the definition of a thick family of geodesics we have
$\sum_id(z_i,\widetilde z_i)\ge\alpha$, these inequalities show
that if we choose $z_i'$ to be the one of $z_i$ and $\widetilde
z_i$, for which $\frac{f(w_i)-f(z_i')}{d(w_i,z_i')}$ and
$\frac{f(z_i')-f(w_{i-1})}{d(z_i',w_{i-1})}$ are more distant from
each other, we have a chance to get the desired condition. (This
is what we verify below.)

We pick $z'_i$ to be $z_i$ if the left-hand side of \eqref{E:z} is
larger than the left-hand side of \eqref{E:zTilde}, and pick
$z'_i=\widetilde z_i$ otherwise.\medskip

Let us estimate $||M_2-M_1||_1$. First we estimate the part of
this difference corresponding to the interval
$\left({d(w_0,w_{i-1})},{d(w_0,w_i)}\right]$. Since the
restriction of $M_2$ to the interval
$\left({d(w_0,w_{i-1})},{d(w_0,w_i)}\right]$ is a two-valued
function, and $M_1$ is constant on the interval, the integral
\begin{equation}\label{E:IntOneInt}\int_{{d(w_0,w_{i-1})}}^{{d(w_0,w_{i})}}||M_2-M_1||dt\end{equation}
can be estimated from below in the following way. Denote the value
of $M_2$ on the first part of the interval by $x$, the value on
the second by $y$, the value of $M_1$ on the whole interval by
$z$, the length of the first interval by $A$ and of the second by
$B$. We have: the desired integral is equal to $A||x-z||+B||y-z||$
and therefore can be estimated in the following way:
\[\begin{split}A||x-z||+B||y-z||&\ge\max\{||x-z||,||y-z||\}\cdot\min\{A,B\}\\&\ge\frac12||x-y||\min\{A,B\}.\end{split}\]

Therefore, assuming without loss of generality that the left-hand
side of \eqref{E:z} is larger than the left-hand side of
\eqref{E:zTilde}, we can estimate the integral in
\eqref{E:IntOneInt} from below by
\[\begin{split}&\frac12\left\|\frac{(f(w_i)-f(z_i))}{d(w_i,z_i)}-\frac{(f(z_i)-f(w_{i-1}))}{d(z_i,w_{i-1})}\right\|
\cdot\min\left\{{d(w_i,z_i)},{d(z_i,w_{i-1})}\right\}
\\&\qquad\ge\frac14\,\ell\, d(z_i,\widetilde z_i)
\left(\frac1{d(w_i,z_i)}+\frac1{d(z_i,w_{i-1})}\right)\cdot\min\left\{{d(w_i,z_i)},{d(z_i,w_{i-1})}\right\}
\\&\qquad\ge\frac14\,\ell\,{d(z_i,\widetilde z_i)}.\end{split}
\]
Summing over all intervals and using the condition
$\sum_{i=1}^{m}d(z_i,\widetilde z_i)\ge\alpha$, we get
$||M_2-M_1||\ge \frac14\ell \alpha$.
\medskip

Now we recall that the last condition of the definition of a thick
family of geodesics implies that
\begin{equation}\label{E:listBefore(2)}w_0,
z'_1, w_1, z'_2, w_2,\dots, z'_m,w_m,\end{equation}  where each
$z'_i$ is either $z_i$ or $\widetilde z_i$ depending on the choice
made above, belongs to some geodesic in the family.

We use all of points in \eqref{E:listBefore(2)} as control points
and find new sequence $\{w_i^2\}_{i=0}^{m_2}$ of common points and
a new sequence of pairs $\{z_i^2,\widetilde z_i^2\}_{i=1}^{m_2}$
with substantial separation:  $\sum_{i=1}^{m_2}d(z_i^2,\widetilde
z_i^2)\ge\alpha$.

We use $\{w_i^2\}_{i=0}^{m_2}$  to construct $M_3$ and the
suitably selected sequence
\[w_0^2, z'^2_1, w_1^2, z'^2_2, w_2^2,\dots, z'^2_{m_2}, w_{m_2}^2\] to
construct $M_4$. We continue in an obvious way. The inequalities
$||M_{2k}-M_{2k-1}||\ge \frac14\ell \alpha$ imply that the
martingale is divergent.
\end{proof}

\subsection{Reflexivity}\label{S:Refl}

\begin{problem}\label{P:TestRefl} Is it possible to characterize the class of
reflexive spaces using test-spaces?
\end{problem}

Some comments on this problem:

\begin{remark} It is worth mentioning that a metric space (or spaces)
characterizing in the described sense reflexivity or the
Radon-Nikod\'{y}m property cannot be uniformly discrete (that is,
cannot satisfy $\inf_{u\ne v}d(u,v)>0$). This statement follows by
combining the example of Ribe \cite{Rib84} of Banach spaces
belonging to these classes which are uniformly homeomorphic to
Banach spaces which do not belong to the classes, and the
well-known fact (Corson-Klee \cite{CK63}) that uniformly
continuous maps are Lipschitz for (nontrivially) ``large''
distances.
\end{remark}

I noticed that combining two of the well-known characterizations
of reflexivity (Pt\'ak \cite{Pta59} - Singer \cite{Sin62} - Pe\l
czy\'nski \cite{Pel62} - James \cite{Jam64b} -
D.~Milman--V.~Milman \cite{MM65}) and some differentiation theory
(Mankiewicz \cite{Man73} -
 Christensen \cite{Chr73} - Aronszajn \cite{Aro76}, see also presentation in \cite{BL00}) we get a purely metric
characterization of reflexivity. This characterization can be
described as a submetric test-space characterization of
reflexivity:

\begin{definition} A \emph{submetric test-space} for a class $\mathcal{P}$ of Banach spaces is defined as a metric space $T$
with a marked subset $S\subset T\times T$ such that the following
conditions are equivalent for a Banach space $X$:

\begin{enumerate}

\item $X\notin\mathcal{P}$.

\item There exist a constant $0<C<\infty$ and an embedding $f:T
\to X$ satisfying the condition

\begin{equation}\label{E:PartBilipsch}
\forall (x,y)\in S\quad d_{T}(x,y)\le ||f(x)-f(y)||\le C
d_{T}(x,y).
\end{equation}
\end{enumerate}

An embedding satisfying \eqref{E:PartBilipsch} is called a
\emph{partially bilipschitz} embedding. Pairs $(x,y)$ belonging to
$S$ are called \emph{active}.
\end{definition}

Let $\Delta \ge 1$. The submetric space $X_\Delta$ is the space
$\ell_1$ with its usual metric. The only thing which makes it
different from $\ell_1$ is the set of active pairs $S_\Delta$: A
pair $(x,y)\in X_\Delta\times X_\Delta$ is active if and only if
\begin{equation}\label{E:DefXDelta} ||x-y||_1\le\Delta||x-y||_s,\end{equation} where $||\cdot||_s$
is the summing norm, that is,

\[||\{a_i\}_{i=1}^\infty||_s=\sup_k\left|\sum_{i=1}^k
a_i\right|.\]

\begin{theorem}[\cite{Ost14a}]\label{T:SubMetRef} $X_\Delta$, $\Delta\ge 2$ is a submetric test space
for reflexivity.
\end{theorem}

The proof goes as follows. Let $Z$ be a non-reflexive space. If
you know the characterization of reflexivity which I meant, you
see immediately that it implies that the space $\ell_1$ admits a
partially bilipschitz embedding into $Z$ with the set of active
pairs described as above.

The other direction. If $\ell_1$ admits a partially bilipschitz
embedding with the described set of active pairs, then the
embedding is Lipschitz on $\ell_1$, because each vector in
$\ell_1$ is a difference of two positive vectors.

Now, if $Z$ does not have the Radon-Nikod\'ym property (RNP), we
are done ($Z$ is nonreflexive). If $Z$ has the RNP, we use the
result of Mankiewicz-Christensen-Aronszajn and find a point of
G\^ateaux differentiability of this embedding. The G\^ateaux
derivative is a bounded linear operator which is ``bounded below
in certain directions''. Using this we can get a sequence in $Z$
which, after application of the non-reflexivity criterion (due to
Pt\'ak - Singer - Pe\l czy\'nski - James - D.\&V.~Milman), implies
non-reflexivity of $Z$. See \cite{Ost14a} for details.

\subsection{Infinite tree property}

See Definition \ref{D:DelTrees} for the definition of  the
infinite tree property. Using a bounded $\delta$-tree in a Banach
space $X$ one can easily construct a bounded divergent $X$-valued
martingale. Hence the infinite tree property implies non-\RNP. For
some time it was an open problem whether the infinite tree
property coincides with non-\RNP. A counterexample was constructed
by Bourgain and Rosenthal \cite{BR80} in the paper mentioned
above. The infinite tree property admits the following metric
characterization.
\medskip

\begin{theorem}[\cite{Ost14a}]\label{T:SubInfTree} The class of Banach spaces with the infinite tree
property admits a submetric characterization in terms of the
metric space $D_\omega$ with the set $S_\omega$ of active pairs
defined as follows: a pair is active if and only if it is a pair
of vertices of a quadrilateral introduced in one of the steps.
\end{theorem}

It would be interesting to answer the following open problem:

\begin{problem} Whether the infinite diamond $D_\omega$ is a
test-space for the infinite tree property?
\end{problem}

\begin{remark} It is worth mentioning that if we restrict our
attention to {\bf dual Banach spaces}, the following three
properties are equivalent:

\begin{itemize}

\item[(1)] Non-\RNP.

\item[(2)] Infinite tree property.

\item[(3)] Bilipschitz embeddability of $D_\omega$.
\end{itemize}

The implication (1) $\Rightarrow$ (2) is due to Stegall
\cite{Ste75}. The implication (2) $\Rightarrow$ (1) follows from
Chatterji \cite{Cha68}. The equivalence of (1) and (3) was proved
in \cite{Ost14a}.\end{remark}

\section{Acknowledgements}

The author gratefully acknowledges the support by NSF DMS-1201269.

The author would like to thank Gilles Lancien for the invitation
to give a series of five lectures at ``Autumn School on Nonlinear
Geometry of Banach Spaces and Applications'' (M\'etabief, France,
October 2014). This paper was prepared on the basis of notes of
that lecture series. The same lecture series was also presented at
the seminar at St. John's University. The author would like to
thank the participants of both the school and the seminar for the
numerous questions, corrections and comments, many of which were
incorporated into the final version of the paper. Also the author
would like to thank Gideon Schechtman and Hannes Thiel for very
useful comments on the first draft of this paper.

\renewcommand{\refname}{

\end{large}

\end{document}